\documentclass{article}[12pt]
\usepackage{graphicx} 

\usepackage[margin=1in]{geometry}
\usepackage{setspace}
\RequirePackage{amsthm,amsmath,amsfonts,amssymb}
\RequirePackage[colorlinks,citecolor=blue,urlcolor=blue]{hyperref}
\RequirePackage{graphicx}

\usepackage{bm}
\usepackage{array}
\usepackage{enumerate}
\usepackage{diagbox}
\usepackage{lscape}
\usepackage{multirow}
\usepackage{natbib}

\theoremstyle{plain}

\newtheorem{theorem}{Theorem}[section]
\newtheorem{lemma}[theorem]{Lemma}

\theoremstyle{remark}
\newtheorem{definition}[theorem]{Definition}
\newtheorem{example}{Example}

\newtheorem{remark}[theorem]{Remark}

\newtheorem{proposition}[theorem]{Proposition}

\newtheorem{assumpA}{Assumption}
\newtheorem{assumpB}{Assumption}
\newtheorem{assumpC}{Assumption}
\newtheorem{assumpD}{Assumption}
\newtheorem{assumpE}{Assumption}

\begin{document}


\begin{center} {\bf\Large  Heavy-tailed $p$-value combinations\\ from the perspective of extreme value theory}\end{center}
	\centerline{\textsc{Yeonwoo Rho$^1$}\footnote{
	 Email: Y. Rho (yrho@mtu.edu)}}
		
		\bigskip
	\centerline {$^1$Michigan Technological University}
	\bigskip
	
\begin{abstract}
Handling multiplicity without losing much power has been a persistent challenge in various fields that often face the necessity of managing numerous statistical tests simultaneously. Recently, $p$-value combination methods based on heavy-tailed distributions, such as a Cauchy distribution, have received much attention for their ability to handle multiplicity without the prescribed knowledge of the dependence structure. This paper delves into these types of $p$-value combinations through the lens of extreme value theory. Distributions with regularly varying tails, a subclass of heavy tail distributions, are found to be useful in constructing such $p$-value combinations. Three $p$-value combination statistics (sum, max cumulative sum, and max) are introduced, of which left tail probabilities are shown to be approximately uniform when the global null is true. The primary objective of this paper is to bridge the gap between current developments in $p$-value combination methods and the literature on extreme value theory, while also offering guidance on selecting the calibrator and its associated parameters.
\\
		\textit{key words}: Extreme value theory; Heavy-tailed distribution; M-merging function; P-value combination.
		\\
\end{abstract}

\section{Introduction}\label{sec:intro}
Multiple comparisons have been a long-standing issue in statistics and related areas. There is a recent surge of interest in addressing multiplicity via $p$-value combinations that are robust to dependence among individual tests. These methods apply a function such as a convex combination to transformed $p$-values. Depending on what function and transformation are used, the random behaviors of the combined statistics under the null vary. This idea traces back to Fisher's method 
\citep{fisher1925} and Stouffer's $Z$-score \citep{stouffer1949american}, which are developed for independent $p$-values. 
\cite{wilson2019harmonic}, \cite{liu2020cauchy}, and \cite{vovk2020combining} have reinvigorated this already popular topic by introducing tools and frameworks that can handle dependent, sometimes perfectly correlated, $p$-values without the need for any special dependence modeling or estimation.

\cite{wilson2019harmonic} and \cite{liu2020cauchy}, while developed independently, share a similar idea. Both consider a convex combination of heavy-tail transformed $p$-values. \cite{wilson2019harmonic} is based on Pareto distribution with shape parameter 1, while \cite{liu2020cauchy} uses the standard Cauchy distribution. \cite{wilson2020generalized} extended \cite{wilson2019harmonic} to allow other shape parameters, and \cite{ling2022stable} extended \cite{liu2020cauchy}'s method to a stable distribution family. \cite{chen2023trade} pointed out the similarity between \cite{wilson2019harmonic} and \cite{liu2020cauchy}. \cite{fang2024heavy} proposed a wider class of transformation functions based on distributions with regularly varying tails. 

Above mentioned works mostly concern the tail probability of combined $p$-values under the null, requiring a smaller significance rate in order to achieve an accurate size of a test. For instance, as \cite{wilson2019harmonic} notes in their equation [5], the tail probability of harmonic mean p (HMP) is approximately uniform as long as this value is small enough, or in other words, in its left tail.   \cite{vovk2020combining} and related papers took a slightly different approach. \cite{vovk2020combining,vovk2021values,vovk2022admissible} built a framework, introducing the notion of ``precise merging functions." Precise merging functions govern the entire domain of support of $p$-values, not just its left tail. \cite{vovk2020combining}'s framework is indeed closely related to the rest of the recent literature in this field, as discussed in \cite{chen2023trade}.

This paper is closely related to the literature with the tail probability approach. As \cite{fang2024heavy} explored, this approach is valid as long as the {\it calibrator}, following \cite{vovk2020combining}'s notion, is related to a distribution function with regularly varying tails, and if the $p$-values from multivariate normal statistics with any dependence structure is considered. However, many interesting questions have not been properly answered yet. For instance, there has been no theoretical guidance on the choice of calibrator or the regular variation index in the literature. It is not clear if this result can be further extended to a wider class of heavy-tailed distributions. The tail probability argument is only valid for small significance levels, but it has not been investigated how small significance levels should be. These methods are so far limited to summations or convex combinations. \cite{vovk2020combining} showed that their framework is also valid for a wider class of merging functions, including order statistics, which establishes a connection to other popular $p$-value combinations such as the Bonferroni and the Simes methods. The close connection between the tail probability approaches and  \cite{vovk2020combining}'s framework suggests a potential breakthrough in this area. 

This paper aims to find answers to some of these questions based on recent advances in extreme value theory. In extreme value theory, it is well-known that the tail probability of the sum of independent heavy-tailed random variables is equivalent to that of the sum of tail probabilities. Similar observations can be made for the maximum and the maximum of the cumulative sums of independent heavy-tailed random variables. In particular, the equivalence of tail probabilities of the sum and of the maximum for heavy-tailed random variables with positive support is called the {\it catastrophe principle} \citep[Definition 3.1]{nair2022fundamentals}, and the {\it subexponential} distributions plays a key role in this relationship \citep[Lemma 3.1]{nair2022fundamentals}. This motivates a search for other forms of merging functions such as maximum as well as a wider class of calibrators beyond the regularly varying distributions. 

The contribution of this paper is threefold. First, we adapt the existing literature on tail probabilities of the sum of dependent random variables with heavy-tailed distributions to the realm of $p$-value combination. Thanks to the recent developments in this field, we could widen the class of calibrators, consider the diverging number of $p$-values case, and introduce random weights.
Second, we adapt \cite{vovk2020combining}'s framework and connected it to the tail probability approach. \cite{vovk2020combining}'s framework has been somewhat distant from, for instance, \cite{fang2024heavy}. This paper establishes a  direct link between the two approaches, extending \cite{vovk2020combining}'s notion of asymptotically precise tests. Third, we offer guidance, informed partly by theory, on selecting calibrators and their associated parameters for different combinations of significance levels and the number of tests.

The rest of this paper is organized as follows. Section \ref{sec:heavytail} introduces the three combination statistics (sum, cumulative sum, and max) inspired by the catastrophe principle in extreme value theory. Some key concepts in extreme value theory are also briefly introduced. Section \ref{sec:subexponential} organizes assumptions and results in the literature on the equivalence of tail probability of dependent and heavy-tailed random variables, which are used to prove the validity of the three combination statistics. Our results allow subexponential calibrators and random weights, which have not yet been introduced in the $p$-value combination literature. The behaviors of the $p$-value combination methods under the presence of tests with perfect or strong positive correlations are also analyzed. Section \ref{sec:regvar_Mmergfun} limits the calibrator class to regularly varying and emphasize the connection between our framework and that of \cite{vovk2020combining}. Section \ref{sec:ninft} proves the validity of the sum-combination test with regularly varying calibrators which is robust to the dependence among tests. The validity of the max-combination test is also discussed but only for an independent case. Section \ref{sec:gammachoice} further leads the discussion to the choice of the index $\gamma$ of a regularly varying calibrator in relation with the significance level $\alpha$. Section \ref{sec:simulation} provides a simulation which explores the finite sample sizes and powers of the $p$-value combination tests. 
All proofs are relegated to the Appendix.


Throughout this paper, we use $\bm{P}$ and $\bm{E}$ to indicate a probability measure and its expectation. The probability spaces the probability measure is based on may vary depending on the context. The maximum among $X_1,\ldots,X_n$ is denoted as $\bigvee_{i=1}^n X_i=\max_{i=1,\ldots,n}X_i$. 
The symbol $f(t)\sim g(t)$ as $t\to\infty$ and $f(\epsilon)\sim g(\epsilon)$ as $\epsilon\to0$ signify $\lim_{t\uparrow\infty}\frac{f(t)}{g(t)}=1$ and $\lim_{\epsilon\downarrow0}\frac{f(\epsilon)}{g(\epsilon)}=1$, respectively. We will often omit the qualifier $t\to\infty$ or $\epsilon\to0$ when it is clear in the context. The set of real numbers is denoted as $\mathbb{R}$. The standard normal cumulative distribution function is denoted as $\Phi(x)=\frac{1}{\sqrt{2\pi}}\int_{-\infty}^x e^{-s^2/2}ds$. The notion $\mathbb{I}_A$ is an indicator function that takes value 1 if the statement $A$ is true and 0 otherwise. The function $\Gamma(z)=\int_0^\infty t^{z-1}e^{-t}dt$ is the usual gamma function.


\section{Heavy-tail combination tests}\label{sec:heavytail}

For any positive integer $n$, consider a vector of uniform random variables $\bm{U}_n=(U_{1;n},\ldots,U_{n;n})$, where $U_{i;n}$ follow a uniform distribution between 0 and 1, are defined on the same probability space, and are possibly dependent. These uniform random variables $U_{i;n}$ represent $p$-values $p_{i;n}$ of the $n$ subtests under the null. If the $i$th null hypothesis $H_{0i}$ is true, $p_{i;n}$ is distributed uniformly between 0 and 1, and we denote $p_{i;n}=U_{i;n}$ to emphasize that these $p$-values are under their respective nulls. Under the global null hypothesis $\bm{H}_{0}$,  that is, if all  $H_{0i}$, $i=1,\dots,n$ are true, all elements in $\bm{p}_n=(p_{1;n},\ldots,p_{n;n})'=\bm{U}_n$ marginally follow a uniform distribution.
Following \cite{vovk2020combining,vovk2021values,vovk2022admissible}, we can understand $U_{i;n}$ be random variables defined on a fixed atomless probability space $(\Omega,\mathcal{A},P)$. It is easy to see that $U_{i;n}$ are {\it p-variables}, since $P(U_{i;n}\leq \epsilon)= \epsilon$ for any $\epsilon\in(0,1)$. 
Let $X$ be a generic random variable with a continuous cumulative distribution function (CDF) $F$; that is, $F(x)=F_X(x)=\bm{P}(X<x)$. We denote the survival function as $\overline{F}(x)=1-F(x)$. Since we limit $F$ be continuous, the inverse cdf ${F}^{-1}(p)$ and the inverse survival function $\overline{F}^{-1}(p)$ are well-defined. 
Define $X_{i;n}=X_{i;n}(p_{i;n})=F^{-1}(1-p_{i;n})=\overline{F}^{-1}(p_{i;n})$ for $i=1,\ldots,n$, and  $\bm{X}_{n}=\bm{X}_{n}(\bm{p}_n)=(X_{1;n},\ldots,X_{n;n})'$. 
 Under the global null,
$\bm{P}(X_{i;n}<x)=P\left(\overline{F}^{-1}(U_{i;n})<x\right)=P\left(U_{i;n}>\overline{F}(x)\right)=F(x)$. We refer $\overline{F}^{-1}(\cdot)$ as the {\it calibrator}, following \cite{vovk2020combining,vovk2021values,vovk2022admissible}.

Let $\bm{w}_n=(w_{1;n},\ldots,w_{n;n})'$ be the  weight vector. From here, we write $X_{i;n}=X_i$, $p_{i;n}=p_i$, $U_{i;n}=U_i$, and $w_{i;n}=w_i$, suppressing the dependence on $n$ when there is no confusion; 
We propose the following three combination statistics:

\begin{align}
    S_{1,F,w}(\bm{p}_{n})&=\sum_{i=1}^nw_{i}X_{i}=\sum_{i=1}^nw_{i}\overline{F}^{-1}(p_{i})
    \\ S_{2,F,w}(\bm{p}_{n})&= \bigvee_{i=1}^n\left(\sum_{j=1}^iw_{j}X_{j}\right)= \bigvee_{i=1}^n\left(\sum_{j=1}^iw_{j}\overline{F}^{-1}(p_{j})\right)
    \\ S_{3,F,w}(\bm{p}_{n})&=\bigvee_{i=1}^nw_{i}X_{i}=\bigvee_{i=1}^nw_{i}\overline{F}^{-1}(p_{i})
\end{align}

The Cauchy combination test by \cite{liu2020cauchy} is based on $S_{1,F,w}(\bm{p}_n)$ with calibrator $\overline{F}^{-1}(p)=
\tan\{(0.5-p)\pi\}$, which is the inverse of standard Cauchy survival function. \cite{liu2020cauchy} assumes that, under the null, $p_i=2-2\Phi(|Y_i|)=U_i$ are uniform when $Y_i$ are marginally standard normal with arbitrary dependence structure. \cite{liu2020cauchy}'s Theorem 1 states that $\bm{P}\left(S_{1,F,w}(\bm{U}_n)>t\right)\sim\sum_{i=1}^nw_i\bm{P}(X>t)$, where $X$ is a standard Cauchy random variable.  
Note that we write $S_{j,F,w}(\bm{U}_n)$ instead of $S_{j,F,w}(\bm{p}_n)$ when emphasizing the given statistic is considered under the global null hypothesis. \cite{fang2024heavy} proved a similar statement for a wider class of $F$ that includes the standard Cauchy, introducing 
\begin{equation}\label{eq:subexponential}\bm{P}(X_1+\cdots+X_n>t)\sim n\bm{P}(X_1>t)\end{equation}
as $t\to\infty$ in their equation (3.1) as ``an important property of regularly varying tailed distributions." Although \cite{fang2024heavy}'s results are limited to the regularly varying class, it is well-known that equation (\ref{eq:subexponential})  holds for an even wider class of heavy-tailed distributions. In fact, equation (\ref{eq:subexponential}) is the very definition of {\it subexponential} distributions. A distribution $F$ on $\mathbb{R}$ is {\it long-tailed}, $F\in\mathcal{L}$, if $\overline{F}(s)>0$ for all $s$ and $\overline{F}(t+s)\sim\overline{F}(t)$ as $t\to\infty$ for any $s>0$.  A distribution $F$ is said to be {\it subexponential} and denoted as $F\in
 \mathcal{S}$ if $F$ is long-tailed and equation (\ref{eq:subexponential}) holds for independent and identically distributed (iid) $X_i$. Most well-known heavy-tailed distributions are subexponential, and therefore, long-tailed. The subexponential $F$ also leads to the {\it catastrophe principle} or the {\it principle of a single big jump}, 
 $\bm{P}\left(S_{1,F,w}(\bm{U}_n)>t\right)\sim
\bm{P}\left(S_{3,F,w}(\bm{U}_n)>t\right)$, where the sum and max have approximately the same tail probabilities. See, for example, Equation (1.6) of \cite{foss2013subexponential} and Definition 3.1 and equation (3.4) of \cite{nair2022fundamentals}.

When $X_i$ are iid and nonnegative, with equal weights $w_1=\cdots=w_n$, it is well-known that for all $j=1,2,3$,
\begin{equation}\label{eq:indep_tailprob}
\bm{P}\left(S_{j,F,w}(\bm{U}_n)>t\right)\sim  \sum_{i=1}^n \bm{P}\left(w_iX_i>t\right).
 \end{equation}
See, for instance, Section 2 of \cite{geluk2006tail}.  \cite{tang2003randomly} further showed that equation (\ref{eq:indep_tailprob}) holds true even when the weights $w_i$ are random, as long as $\{X_i\}_{i=1}^n$ are independent. Understanding tail probabilities of $S_{j,F,w}(\bm{U}_{n})$ with heavy-tailed $F$ and dependent $X_i$ has been an active research topic. We will utilize this literature to advance theoretical developments of $p$-value combination methods in this paper.

We finish this section with a few other key concepts. A distribution $F$ is {\it dominatedly varying} tailed, $F\in\mathcal{D}$, if 
$$\limsup_{t\to\infty}\frac{\overline{F}(at)}{\overline{F}(t)}<\infty$$
for some $0<a<1$. An important subclass of $\mathcal{D}\cap\mathcal{L}$ is {\it consistently varying} tailed distributions, denoted as $F\in\mathcal{C}$, if 
$$\lim_{a\uparrow1}\limsup_{t\to\infty}\frac{\overline{F}(at)}{\overline{F}(t)}=1,~~~\mbox{or equivalently}~~\lim_{a\downarrow1}\liminf_{t\to\infty}\frac{\overline{F}(at)}{\overline{F}(t)}=1.$$
In particular, if the limit of $\frac{\overline{F}(at)}{\overline{F}(t)}$ exists, the related theory can be more intuitive. This motivates understanding the {\it scale invariance} and {\it regularly varying tails}.
A
 distribution $F$ is {\it scale invariant} if there exists $t_0>0$ and a continuous positive function $g$ such that $$\frac{\overline{F}(at)}{\overline{F}(t)}=g(a)$$ for all $t$ and $a$ that satisfy $t>t_0$ and $a t>t_0$.
A distribution $F$ is {\it asymptotically scale invariant} if there exists a continuous positive function $g$ such that $$\frac{\overline{F}(at)}{\overline{F}(t)}\sim g(a)$$ as $t\to\infty$ for any $a>0$ \citep[Definitions 2.1 and 2.2]{nair2022fundamentals}.
The only scale invariant distribution has a power-law tail such that $\overline{F}(t)=ct^{-\gamma}$ for large enough $t$ with $c\geq0$ and $\gamma>0$ \citep[Theorem 2.1]{nair2022fundamentals}. For instance, a Pareto distribution,  $\overline{F}(t)=(t/t_m)^{-\gamma}$, $t>t_m$, with scale parameter $t_m>0$ and shape parameter $\gamma>0$, is scale invariant. We will set the scale parameter $t_m=1$ from the Pareto $F$ in this paper for brevity.
A distribution $F$ defined on $\mathbb{R}$ is {\it regularly varying at infinity with index $\gamma>0$} if
\begin{equation}\label{eq:regvardef}\frac{\overline{F}(at)}{\overline{F}(t)}\sim a^{-\gamma}\end{equation}
as $t\to\infty$ for any $a>0$ \citep[page 32]{foss2013subexponential}, and we write $F\in \mathcal{RV}_{-\gamma}$.
It is trivial that regularly varying distributions are asymptotically scale invariant. In addition, the asymptotically scale invariant distributions should have regularly varying tails. That is, if the ratio $\frac{\overline{F}(at)}{\overline{F}(t)}$ is controlled at its tail, this asymptotic ratio should follow a power law in $a$.  See, for instance, Proposition 0.4 of \cite{resnick2008extreme} and Theorem 2.2 of \cite{nair2022fundamentals}.  Extending the regular variation slightly, if, for any $a>1$,  
$$a^{-\beta}\leq \liminf_{t\to\infty}\frac{\overline{F}(at)}{\overline{F}(t)}\leq\limsup_{t\to\infty}\frac{\overline{F}(at)}{\overline{F}(t)}\leq a^{-\alpha}$$
for some $0\leq \alpha\leq \beta<\infty$,
we say $F$ has {\it extended regularly varying} tail and write $F\in \mathcal{ERV}(-\alpha,-\beta)$.
Note that $\mathcal{RV}_{-\gamma}\subset \mathcal{ERV}\subset\mathcal{C}\subset \mathcal{D}\cap\mathcal{L}\subset\mathcal{S}\subset\mathcal{L}$.

It is well-known that $F\in \mathcal{RV}_{-\gamma}$ satisfies the {\it balance condition}: $\overline{F}(t)\sim p_Ft^{-\gamma}L(t)$ and $F(-t)\sim q_Ft^{-\gamma}L(t)$ as $t\to\infty$, where $p_F+q_F=1$ and $p_F,q_F\geq0$. Here, the function $L(t)$ satisfies 
$\lim_{t\to\infty}\frac{L(at)}{L(t)}=1$ for any $a>0$ and is said to be {\it slowly varying at infinity}. 
Examples of slowly varying functions include a function with a finite limit $\lim_{x\to\infty}L(x)=l_F\in(0,\infty)$ but is not limited to this class of functions. For instance, $L(x)=\log(x)$ does not converge to a finite limit but still is slowly varying.
\cite{fang2024heavy} directly assumes the balance condition, for their additive $p$-value combination tests, as the condition for the regularly varying distributions in their assumption (A3).

\section{Subexponential Calibrators with a Finite Number of Tests}\label{sec:subexponential}
 

\cite{fang2024heavy} showed that, extending \cite{liu2020cauchy}'s result, the size of the global test based on  the sum-combination statistic $S_{1,F,w}(\bm{U}_n)$ is well-controlled at its extreme tail, for any fixed number of tests $n$  if $F$ is regularly varying.
In this section, we demonstrate that \cite{fang2024heavy}'s claim can be further extended to subexponential calibrators, as well as to the max-combination statistic  $S_{3,F,w}(\bm{U}_n)$ and to the max-cumulative-sum combination statistic $S_{3,F,w}(\bm{U}_n)$. 

Recall that our calibrator is defined as an inverse of a survival function $\overline{F}^{-1}(\cdot)$. We often refer to $F$ as  {\it calibrator} as well for convenience, although, technically, $F(\cdot)$ is one minus the inverse of a calibrator; $F(\cdot)=1-\overline{F}(\cdot)$. We  assume that $F$ is continuous so that there are always well-defined one-to-one relations among $F$, $\overline{F}$, $F^{-1}$, and $\overline{F}^{-1}$ for a simpler exposition in theory.
Our results are based on recent advances in tail probability literature and can accommodate random weights, which has not yet been allowed in $p$-value combination literature. The following assumptions on the weights $\{w_{1},\ldots,w_{n}\}$ are often considered in the literature.
\begin{assumpA}\label{assump:w_equal}
The weights $\{w_{1},\ldots,w_{n}\}$  take equal values: that is, $w_{i}=1/n$ for all $i=1,\ldots,n$.
\end{assumpA}
\begin{assumpA}\label{assump:w_nonrandom}
The weights $\{w_{1},\ldots,w_{n}\}$ are nonrandom with $\sum_{i=1}^nw_{i}=1$ and $w_{i}\in(0,1)$. 
\end{assumpA}
\begin{assumpA}\label{assump:w_random}
The weights $\{w_{1},\ldots,w_{n}\}$ are possibly random with 
  $\bm{P}(0<w_{i}<1)=1$ and $\sum_{i=1}^n\bm{E}w_{i}=1$. The weights $\{w_{i}\}_{i=1}^n$ are independent from  the calibrated $p$-values $\{X_{i}\}_{i=1}^n$.
\end{assumpA}
It is evident that Assumption \ref{assump:w_equal} is a special case of Assumption \ref{assump:w_nonrandom}, and likewise, Assumption \ref{assump:w_nonrandom} is a special case of Assumption \ref{assump:w_random}. It is worth noting that Assumption \ref{assump:w_random} imposes no limitations on the strength or form of the dependence among $\{w_i\}_{i=1}^n$, as long as the weights are chosen independently from $\{X_i\}_{i=1}^n$. Conversely, the dependence structure of $\{X_i\}_{i=1}^n$ does require some limitations.

We introduce a few dependence conditions on  $\{X_i\}_{i=1}^n$ that are often adapted in the literature. \cite{sibuya1960bivariate} introduced the {\it upper tail asymptotic independence (UTAI)} and has played a crucial role in the asymptotic tail equivalence literature. 
\begin{assumpB}\label{assmp:UTAI}
The calibrated $p$-values $\{X_i\}_{i=1}^n$ are upper tail asymptotically independent (UTAI): that is, for $1\leq i\neq j\leq n$, 
$$\lim_{x\to\infty}\bm{P}(X_i>x|X_j>x)=0.$$
\end{assumpB}
\cite{davis1996limit} 's Lemma 1 and \cite{chen2009sums}'s results assume UTAI. In particular, \cite{fang2024heavy}, in their supplementary material S1, proved that with their multivariate normal assumption, the calibrated $p$-values with regularly varying tails are UTAI. 

\cite{geluk2009asymptotic} introduced a similar but slightly stronger concept, {\it tail asymptotic independence (TAI)}, which are used by some subsequent literature such as \cite{li2013pairwise} and \cite{cang2017extremal}.
\begin{assumpB}\label{assmp:TAI}
The calibrated $p$-values $\{X_i\}_{i=1}^n$ are tail asymptotically independent (TAI): that is, for $1\leq i\neq j\leq n$,
$$\lim_{\min(x_i,x_j)\to\infty}\bm{P}(|X_i|>x_i|X_j>x_j)=0.$$
\end{assumpB}

While Assumptions \ref{assmp:UTAI} and \ref{assmp:TAI} cover a wide range of dependence structure, as pointed out by \cite{geluk2009asymptotic}, the UTAI and TAI assumptions are not enough to extend $F$  beyond dominatedly varying tailed distributions. In order to cover important heavy-tailed distributions such as Weibull and log-normal that are subexponential but not dominatedly varying, \cite{geluk2009asymptotic} and \cite{ko2008sums} proposed the following assumptions.
\begin{assumpB}[\cite{geluk2009asymptotic}]\label{assmp:geluktang}
The calibrated $p$-values $X_i$ have distribution functions $F_i$, respectively for all $i=1,\ldots,n$. For any $n\geq2$, there exists positive constants $t_0=t_0(n)$ and $M=M(n)$ such that, for all $i=1,\ldots,n$, $x_i>t_0$, and $x_j>t_0$ with $j\in J$,
$$\bm{P}(|X_i|>x_i|X_j=x_j~{\rm with}~j\in J)\leq M\overline{F}_i(x_i),$$
where $J=\{1,\ldots,n\}/\{i\}$.
\end{assumpB}
\begin{assumpB}[\cite{ko2008sums}]\label{assmp:kotang}
Let $X_i^+=\max(X_i,0)$, where $X_i$ are the calibrated $p$-values.
There exist positive constants $t_0=t_0(n)$ and $M=M(n)$ such that, for all $t>t_0$ and $j\in\{2,\ldots,n\}$,
$$\sup_{t_0\leq x\leq t}\frac{\bm{P}(\sum_{i=1}^{j-1}X_i^+>t-x|X_j=x)}{\bm{P}(\sum_{i=1}^{j-1}X_i^+>t-x)}\leq M.$$
\end{assumpB}

When $n=2$, these two conditions are essentially equivalent. However, when $n>2$, there is no clear relationship between Assumptions \ref{assmp:geluktang} and \ref{assmp:kotang}, and are often used in parallel in literature, as can be seen in, for instance, \cite{jiang2014max} and \cite{gao2019asymptotic}. In particular, Example 2.1 of \cite{jiang2014max} demonstrated that Assumption \ref{assmp:kotang} does not imply \ref{assmp:geluktang}.

While Assumption \ref{assmp:kotang} does not allow too strong positive dependence, Assumption \ref{assmp:kotang} covers a wide range of dependence structures such as stochastically decreasing or through a copula structure with a bounded second-order mixed derivative. See Remarks 2.1, 2.3, and 2.4 of \cite{ko2008sums} for more details. Remark 2.2 of \cite{ko2008sums} states that Assumption \ref{assmp:kotang} implies the upper tail asymptotic independence (UTAI), but not the other way around. 

Now we compare the above assumptions with those in the $p$-value combination literature.
\cite{vovk2020combining} and \cite{vovk2022admissible} allow any kind of dependence, which, obviously, is much weaker than Assumptions \ref{assmp:UTAI}--\ref{assmp:kotang}. These assumptions in Assumptions \ref{assmp:UTAI}--\ref{assmp:kotang}, however, are somewhat comparable to those of \cite{liu2020cauchy} and 
 \cite{fang2024heavy}, which assume that the $p$-values are from $Z$-tests, $p_i=2-2\Phi(|Y_i|)$, where $Y_i$ are from a multivariate normal distribution with an arbitrary correlation matrix.
In particular, \cite{fang2024heavy}'s proof relies on the UTAI property (Assumption \ref{assmp:UTAI}) of multivariate normality, meaning that they actually need a weaker condition than what they proposed in their paper. While this $Z$-test setup satisfies Assumption \ref{assmp:UTAI}, as shown in \cite{fang2024heavy}, this $Z$-test setup does not satisfy Assumptions \ref{assmp:geluktang} nor \ref{assmp:kotang}.
The following theorem proves such a result when $n=2$.
\begin{theorem}\label{thm:bivariate_normal}
     Let $U_1=2\overline{\Phi}(|Y_1|)$ and $U_2=2\overline{\Phi}(|Y_2|)$, where $(Y_1,Y_2)'$ follows a bivariate normal distribution with a mean vector of $(0,0)'$ and the covariance matrix $\Sigma=\bigg(\begin{array}{ll} 1&\rho\\ \rho&1\end{array}\bigg)$ with $0<    |\rho|<1$. Here, $\overline{\Phi}(y)=1-\Phi(y)=\int_{y}^\infty\phi(x)dx$, where $\phi(x)=\frac{1}{\sqrt{2\pi}}\exp(-x^2/2)$ is the standard normal density. Let $X_1=\overline{F}^{-1}(U_1)$ and $X_2=\overline{F}^{-1}(U_2)$. As long as $F$ satisfies $\overline{F}(c)\neq 0$ for any $c>0$,  $\{X_1,X_2\}$ cannot satisfy Assumptions \ref{assmp:geluktang} nor \ref{assmp:kotang}.
\end{theorem}
Considering the necessity of a more strict assumption than the UTAI when $F$ is subexponential, 
Theorem \ref{thm:bivariate_normal} provides one reason to stay with the distributions with regularly varying tails when designing a combination test.

We need one additional assumption to control the left tail behavior of $F$.
It is well-known that $$\mathcal{H}({F})=\left\{h~{\rm on}~[0,\infty):~h(t)\uparrow\infty,~\frac{h(t)}{t}\downarrow0~{\rm and}~ \overline{F}(t- h(t))\sim \overline{F}(t)\right\}$$
is a not empty if $F$ is long-tailed \citep{foss2013subexponential,geng2019tail,qian2022tail}.
\begin{assumpC}\label{assump:leftail} For any $h\in\mathcal{H}({F})$,
$$\lim_{t\to\infty}\sup_{x\in(0,b]}\frac{F(-h(t)/x)}{\overline{F}(t/x)}=0.$$
\end{assumpC}
This assumption requires left tails to be light enough. Any $F$ defined on the positive half line $\mathbb{R}^+$ satisfies Assumption \ref{assump:leftail} since $F(-h(t)/x)=0$ for any $t\geq 0$. However, a distribution with a heavier left tail may not satisfy Assumption \ref{assump:leftail}. 

\begin{remark}\label{remark:assumption_lefttail}{\rm
If $F$ belongs to a regularly varying distribution family, it is simple to check Assumption \ref{assump:leftail}. Recall that 
 a regularly varying distribution with index $\gamma>0$, $F\in \mathcal{RV}_{-\gamma}$, satisfies the balance condition: $\overline{F}(t)\sim p_Ft^{-\gamma}L(t)$ and $F(-t)\sim q_Ft^{-\gamma}L(t)$ as $t\to\infty$, where $p_F+q_F=1$ and $p_F,q_F\geq0$ and $L(t)$ is slowly varying.
 Using the balance condition,
$$\frac{F(-h(t)/x)}{\overline{F}(t/x)}\sim\frac{q_Fh(t)^{-\gamma}x^{\gamma}L(h(t)/x)}{p_Ft^{-\gamma}x^\gamma L(t/x)}
=\frac{q_F}{p_F}\left(\frac{t}{h(t)}\right)^{\gamma}\uparrow \infty$$ as $t\to\infty$ unless $q_F=0$.
This means that, while a regularly varying calibrator is not required to have positive support, it should at least have a very thin left tail to control the size properly under the subexponential umbrella. As mentioned in the Remark \ref{remark:lefttailass_powerside} below, it has been already observed that having $F$ with a thin left tail generally improves the power of the combined test. Assumption \ref{assump:leftail} provides an additional argument to favor  $F$ with as thin left tails as possible from a different angle. In addition, limiting too large negative values is also necessary for size control when we consider diverging $n$, as can be seen from Theorems \ref{thm:ninfty_chen},\ref{thm:ninfty_gao}, and \ref{thm:ninfty_geng} in Section \ref{sec:ninft}.
}\end{remark}

\begin{remark}\label{remark:lefttailass_powerside}{\rm
The necessity of a thin-left tail, similar to Assumption \ref{assump:leftail}, from the power perspective is already discussed in the literature.  A stable distribution with stability parameter $\gamma\in(0,2)$ and skewness parameter $\beta\in[-1,1]$ is $\mathcal{RV}_{-\gamma}$ with $p_F=1+\beta$ and $q_F=1-\beta$ \citep[Theorem 1.2]{nolan2020univariate}. \cite{ling2022stable} argued from their simulation that $\beta=1$ always has better power than $\beta<1$ when $\gamma$ is hold constant. \cite{fang2024heavy} made a similar observation for Cauchy-based calibrators. A Cauchy distribution belongs to $\mathcal{RV}_{-1}$ with a relatively heavier left tail, $p_F=q_F=0.5$, which does not satisfy Assumption \ref{assump:leftail}. \cite{fang2024heavy} observed the ``large negative penalty issue,"  which caused a numerical issue and loss of power. In particular, \cite{fang2024heavy} proposed the truncated Cauchy, for which the left tail is arbitrarily thinned out. The truncated Cauchy first chooses a threshold $0<\delta<0.5$ and sets any $p$-values greater than $1-\delta$ to be equal to $1-\delta$. Their calibrated p-values $X_i=\overline{F}^{-1}(p_i)$ do not follow a standard Cauchy, but there exists $t_\delta>0$ such that $F(-h(t)/x)=0$ for any $t>t_\delta$. The truncated Cauchy satisfies Assumption \ref{assump:leftail}, but involves an additional user-chosen parameter $\delta$.
}\end{remark}

Table \ref{table:literature} summarizes some asymptotic tail equivalence results involving dependent $X_i$ in the literature.

\begin{center}
\begin{table}
\begin{tabular}{c|m{20em}|c|c|c|c } 
 \hline\hline
$n$&Source & $j$ &$w_i$ & Dep. &  $F$ \\\hline 

\multirow{11}{*}{finite}&Lemma 2.1, \cite{davis1996limit} & 1 &\ref{assump:w_equal} &\ref{assmp:UTAI}& $F\in\mathcal{RV}_{-\gamma}$, p.s.\\ 
&Theorem 3.1, \cite{geluk2006tail} & 1,2,3 &\ref{assump:w_equal} & NA& $F\in\mathcal{D}\cap\mathcal{L}$, p.s$.^*$ \\ 
&Theorem 3.1(i), \cite{ko2008sums} & 1,2,3 &\ref{assump:w_equal} &\ref{assmp:kotang}& $F\in\mathcal{S}$, p.s.\\ 
&Theorem 3.1, \cite{geluk2009asymptotic} & 1 &\ref{assump:w_equal} &\ref{assmp:TAI}& $F\in\mathcal{D}\cap\mathcal{L}$\\ 
&Theorem 3.2, \cite{geluk2009asymptotic} & 1&\ref{assump:w_equal}& \ref{assmp:geluktang}& $F\in\mathcal{S}$\\ 

& Theorems 3.1, 3.2, \cite{chen2009sums} & 1&\ref{assump:w_equal}/\ref{assump:w_random}& \ref{assmp:UTAI}& $F\in\mathcal{C}$  \\
& Theorems 2.1, 2.3, \cite{li2013pairwise} & 1&\ref{assump:w_nonrandom}/\ref{assump:w_random}& \ref{assmp:TAI}& $F\in\mathcal{D}\cap\mathcal{L}$ \\
& Theorems 1.1, 1.2, \cite{jiang2014max}
&1,3&\ref{assump:w_equal}&\ref{assmp:kotang}&$F\in\mathcal{L}$,$F*F\in\mathcal{S}$\\
&  Theorem 1.1, \cite{cang2017extremal} & 1&\ref{assump:w_random}& \ref{assmp:TAI}& $F\in\mathcal{D}\cap\mathcal{L}$ \\
& Theorem 3.1, \cite{geng2019tail} & 1,2,3&\ref{assump:w_random}& \ref{assmp:kotang}& $F\in\mathcal{S}$, \ref{assump:leftail}\\
&  Theorem 1, \cite{qian2022tail} & 1,2,3&\ref{assump:w_random}& LWQD & $F\in\mathcal{S}^
*$, \ref{assump:leftail}\\
  \hline

\multirow{5}{*}{infinite}& Theorems 2.1, 2.2, \cite{liu2009precise} & 1&\ref{assump:w_equal}& END& $F\in\mathcal{C}$+alpha\\
& Theorem 3.3, \cite{chen2009sums} & 1&\ref{assump:w_random},\ref{assump:wi_ninfty_chen}& \ref{assmp:UTAI}& $F\in \mathcal{ERV}$, p.s. \\
&  Theorem 2, \cite{bae2017note} & 1&\ref{assump:w_nonrandom},\ref{assump:wi_ninfty_gao}& \ref{assmp:kotang}& $F\in\mathcal{RV}_{-\gamma}$, p.s.\\
&   Theorem 2.1, \cite{gao2019asymptotic} & 1&\ref{assump:w_nonrandom},\ref{assump:wi_ninfty_gao}& \ref{assmp:geluktang}/\ref{assmp:kotang}& $F\in\mathcal{C}$+alpha\\
&  Theorem 3.2, \cite{geng2019tail} & 1,2&\ref{assump:w_random},\ref{assump:wi_ninfty_yi}& \ref{assmp:kotang}& $F\in\mathcal{C}$, \ref{assump:leftail}\\
  \hline
  \hline
\end{tabular}\caption{Summary of asymptotic tail equivalence results in the literature. The ``$j$" column indicates which combination statistic is available for the tail equivalence in (\ref{eq:indep_tailprob}) with $j=1$ being the sum, $j=2$ the maximum cumulative sum, and $j=3$ the maximum; ``$w_j$" column indicates the assumptions on the weights; ``Dep." column indicates the assumptions for the dependence among $X_i$; $F$ column presents the assumptions on  $F$; ``p.s." indicates that $F$ has positive support, in which case, $j=1$ and $j=2$ are identical; 
``p.s$.^*$" indicates that $F$ is defined on $(c,\infty)$ for some $c>-\infty$; $F*G$ indicates the convolution of two distribution functions $F$ and $G$; $S^*$ represents the strong subexponential class; ``NA," ``LWQD," and ``END" refer to the dependence structures introduced in the corresponding papers. In the last three columns,  the ``or'' conditions are connected with the symbol $/$, and the ``and'' conditions are connected with a comma.  While some of these references allow $X_i$ to have different distribution functions $F_i$, this table assumes all $F_i$ are the same, in an effort to make this summary as concise as possible. This table provides an overall view of the recent developments in this area, without the intention of providing an exhaustive list. }\label{table:literature}
\end{table}
\end{center}

Among the results in Table \ref{table:literature}, in the theorem below, we summarize
Theorem 3.2 of \cite{geluk2009asymptotic} and Theorem 3.1 of \cite{geng2019tail}.
Theorem \ref{thm:inS} can be understood as an extension of Corollary 2 of \cite{fang2024heavy} to the subexponential class of calibrators $F$, under a stricter condition on the dependence among $X_i$, Assumption \ref{assmp:geluktang} or \ref{assmp:kotang}. 
\begin{theorem}\label{thm:inS} Let  $F\in\mathcal{S}$. For any finite integer $n$, 
\begin{enumerate}[{\normalfont (i)}]
    \item  the equation (\ref{eq:indep_tailprob}) holds for $j=1$, if Assumptions \ref{assump:w_equal} and \ref{assmp:geluktang} are satisfied.
    \item\label{thm:inS-2} the equation (\ref{eq:indep_tailprob}) holds for $j=1,2,3$, if Assumptions \ref{assump:w_random}, \ref{assmp:kotang}, and \ref{assump:leftail} are satisfied.
\end{enumerate}
\end{theorem}

While the tail equivalence results in Table \ref{table:literature} and Theorem \ref{thm:inS} seem useful,  applying these results to testing is not as straightforward unless all weights are equal or the calibrators are in the regularly varying family.
This is due to the possible difficulty in finding $t$ that satisfies $\alpha=\sum_{i=1}^n\overline{F}(w_i^{-1}t)$ for a given significance level $\alpha$, unless the right tail approximately follows a power law such as Pareto or other regularly varying distributions. We will discuss in Section \ref{sec:regvar_Mmergfun} about how to conduct a test when the calibrator is in the regularly varying class. We would also like to note that simulating these tail probabilities with independent $U_i$, based on (\ref{eq:indep_tailprob}), does not work, particularly when $\alpha$ or $n$ is large. This is due to the poor tail approximations in such cases, as can be seen from the Table \ref{table:size} in Section \ref{sec:simulation}.

We now assume equal weights, or Assumption \ref{assump:w_equal}, for the rest of this section to obtain an analytic form of the critical value to learn the behavior of the combined statistics with subexponential calibrators. When $w_{i}=1/n$, under the global null, $\bm{P}(w_iX_i>t)=\bm{P}(X_i>nt)=\overline{F}(nt)$ for all $i=1,\ldots,n$. The equation (\ref{eq:indep_tailprob}) can be rewritten as
$
     \bm{P}\left(S_{j,\mathcal{S},\ref{assump:w_equal}}(\bm{U}_n)>t\right)
     \sim n\overline{F}(nt)
$ as $t\to\infty$.
The critical values of the combined test with a nominal level $\alpha$ is, therefore, 
\begin{equation}\label{equation:talpha}
    t_{\alpha,n}=\frac{1}{n}{F}^{-1}\left(1-\frac{\alpha}{n}\right)=\frac{1}{n}\overline{F}^{-1}\left(\frac{\alpha}{n}\right).
\end{equation}

\begin{remark}\label{remark:maxcomstat}{\rm
The max-combination statistic with equal weights, $S_{3,\mathcal{S},\ref{assump:w_equal}}(\bm{U}_n)$, with the critical value $t_{\alpha,n}$ chosen as in equation (\ref{equation:talpha}), have the same tail probability no matter what $F$ is used. In fact, the max-combination statistic based on a subexponential calibrator with equal weights is equivalent to Bonferroni's method. This is trivial to see by noticing that, since $\overline{F}$ is a nonincreasing function,
$$\begin{array}{lll}\bm{P}\left(S_{3,\mathcal{S},\ref{assump:w_equal}}(\bm{U}_n)>t_{\alpha,n}\right)&=&\bm{P}\left(\frac{1}{n}\max_{i=1,\ldots,n}\overline{F}^{-1}(U_i)>\frac{1}{n}\overline{F}^{-1}\left(\frac{\alpha}{n}\right)\right)
\\&=&\bm{P}\left(\min_{i=1,\ldots,n}U_i<\frac{\alpha}{n}\right).
\end{array}$$
Theorem \ref{thm:inS} (\ref{thm:inS-2}) validates Bonferroni's method for any finite $n$, with the tail-independence condition, as long as the significance level is small enough: $\bm{P}\left(\min_{i=1,\ldots,n}U_i<\frac{\alpha}{n}\right)\sim\alpha$ as $\alpha\to0$.
}\end{remark}


As we will see in simulation results presented in Table \ref{table:size} in Section \ref{sec:simulation}, there are severe under- or over-rejections for certain choices of calibrators when tests are strongly dependent. Severe over-rejection implies that family-wise error rates are not properly controlled, and strong under-rejection naturally leads to a significant loss of power. These severe under- and over-rejection behaviors can be explained by a careful examination of a case with some perfectly dependent $p$-values included in the sample, similar to Theorem 2 of \cite{fang2024heavy}.
Theorem \ref{thm:inS} can easily be extended to a similar case.  Let $\mathcal{P}=\{i=1,\ldots,n;U_i=U_j~{\rm for~some}j\neq i\}$ be the indices of $p$-values that has a perfect correlation with another $p$-value under considerations. For simplicity, assume that $U_i=U_j$ for any $i,j\in\mathcal{P}$, similarly to the setting in Theorem 2 of \cite{fang2024heavy}. This means that if there are perfectly correlated cases, there will be only one set of $p$-values that are perfectly correlated with each other.
Suppose the number of elements in $\mathcal{P}$ be $n-m$,
where $m$ is an integer $m\geq0$ that represents the number of $p$-values that do not have a perfectly correlated copy within $\{1,\ldots,n\}$ . Without loss of generality, let $\mathcal{P}=\{n-m+1,\ldots,n\}$.  If $n-m\geq 2$, $\mathcal{P}$ is not an empty set. The following theorem extends part (ii) of Proposition \ref{thm:inS}.  Similar results can also be stated in different settings based on Table \ref{table:literature}, which will be omitted for brevity.

\begin{theorem}\label{thm:inS_rho1}
     Let the calibrated $p$-values $X_1,\ldots,X_{m+1}$ satisfy Assumption \ref{assmp:kotang}, the weights $w_1,\ldots,w_{n}$ satisfy Assumption \ref{assump:w_random}, and the calibrator $F$ satisfy Assumption \ref{assump:leftail}.  Then
    \begin{equation}\label{eq:thm:inS:1}\bm{P}(S_{1,F,w}(\bm{U}_n)>t)\sim \sum_{i=1}^{m}\bm{P}\left(w_iX_i>t\right)+\bm{P}\left(\left(\textstyle\sum_{i=m+1}^nw_i\right)X_{m+1}>t\right)\end{equation}
    and
    \begin{equation}\label{eq:thm:inS:3}\bm{P}(S_{3,F,w}(\bm{U}_n)>t)\sim \sum_{i=1}^{m}\bm{P}\left(w_iX_i>t\right)+\bm{P}\left(\left(\vee_{i=m+1}^nw_i\right)X_{m+1}>t\right).\end{equation}  
\end{theorem}

The presence of perfectly correlated $p$-values makes the difference between the sum- and max- combination tests in the way that their weights are combined, which ultimately leads to the difference in their robustness in a strongly correlated case. To demonstrate this behavior, we consider equal weights and regularly varying calibrators. Equations (\ref{eq:thm:inS:1}) and (\ref{eq:thm:inS:3}) in Theorem \ref{thm:inS_rho1} can be rewritten as, as $t\to\infty$,
\begin{equation}\label{eq:thm:inS:11}\bm{P}(S_{1,F,\ref{assump:w_equal}}(\bm{U}_n)>t)\sim m\overline{F}(nt)+\overline{F}\left(\frac{n}{n-m}t\right)\sim \frac{\{m+(n-m)^\gamma\}\overline{F}(t)}{n^{\gamma}}\end{equation}
and 
\begin{equation}\label{eq:thm:inS:33}\bm{P}(S_{3,F,\ref{assump:w_equal}}(\bm{U}_n)>t)\sim m\overline{F}(nt)+\overline{F}(nt)\sim\frac{(m+1)\overline{F}(t)}{n^{\gamma}},\end{equation}
respectively.
On the contrary, the tail probability without any perfect correlation from equation (\ref{eq:indep_tailprob}) is 
\begin{equation}\label{eq:thm:inS:0}\bm{P}(S_{j,F,\ref{assump:w_equal}}(\bm{U}_n)>t)\sim n\overline{F}(nt)=\frac{n\overline{F}(t)}{n^{\gamma}},\end{equation}
which we will use as a reference distribution without any prior knowledge of the existence of perfect or very strong correlation.
The three tail probabilities have the following relationship:  for any $n-m\geq 2$, and $m\geq 0$, when $\gamma\in(0,1)$, 
$$(\ref{eq:thm:inS:33})<(\ref{eq:thm:inS:11})<(\ref{eq:thm:inS:0}),$$
when $\gamma=1$,
$$(\ref{eq:thm:inS:33})<(\ref{eq:thm:inS:11})=(\ref{eq:thm:inS:0}),$$
and when $\gamma>1$,$$(\ref{eq:thm:inS:33})<(\ref{eq:thm:inS:0})<(\ref{eq:thm:inS:11}).$$
This relationship instructs that with a regularly varying calibrator and equal weights, under the presence of perfect or very strong positive dependence, (i) the max-combination statistic always under-rejects, (ii) the sum-combination statistic with $\gamma>1$ cannot control the size, (iii)  the sum-combination statistic with $\gamma=1$ has asymptotically exact size for a small enough significance level, and (iv) the sum-combination statistic with $\gamma\in(0,1)$ always under-rejects. In particular, the under-rejection of a sum-combination statistic gets severer as $\gamma\downarrow0$ and becomes ignorable as $\gamma\uparrow1$. This behavior is natural consequence of the fact that the sum-combination with $\gamma\downarrow0$ is equivalent to the max-combination statistic.


 Now we try to understand the behaviors of the sum- and max-combination statistics under the global null with a wider range of calibrators, beyond the regularly varying class. For simplicity, we assume $m=0$, in which case, $U_1=\ldots=U_n$ for all $n$ and Theorem \ref{thm:inS_rho1} holds trivially without any approximations. 
 The combination statistics can be rewritten in simpler forms, $S_{1,\mathcal{S},\ref{assump:w_equal}}(\bm{U}_n)=\overline{F}^{-1}(U_1)$ and $S_{3,\mathcal{S},\ref{assump:w_equal}}(\bm{U}_n)=\frac{1}{n}\overline{F}^{-1}(U_1)$. The tail probability for the max-combine test with $t_{\alpha,n}=\frac{1}{n}\overline{F}^{-1}(\alpha/n)$ as the threshold is
$$\bm{P}(S_{3,\mathcal{S},\ref{assump:w_equal}}(\bm{U}_n)>t_{\alpha,n})=\bm{P}\left(\frac{1}{n}\overline{F}^{-1}(U_1)>\frac{1}{n}\overline{F}^{-1}\left(\frac{\alpha}{n}\right)\right)=\bm{P}\left(U_1<\frac{\alpha}{n}\right)=\frac{\alpha}{n},$$
which is much smaller than the nominal level $\alpha$ when $n$ is large. This explains the under-rejection behavior of Bonferroni's method for dependent tests.
For the sum-combination statistic, the tail probability
$$\bm{P}(S_{1,\mathcal{S},\ref{assump:w_equal}}(\bm{U}_n)>t_{\alpha,n})=\bm{P}\left(\overline{F}^{-1}(U_1)>\frac{1}{n}\overline{F}^{-1}\left(\frac{\alpha}{n}\right)\right)=\overline{F}\left(\frac{1}{n}\overline{F}^{-1}\left(\frac{\alpha}{n}\right)\right)$$
depends on the choice of $F$. The following examples discuss the theoretical sizes of $S_{1,\mathcal{S},\ref{assump:w_equal}}(\bm{U}_n)$ with popular subexponential $F$.

\begin{example}\label{example:regularlyvarying2}{\rm If $F$ is regularly varying with index $\gamma>0$,
$\overline{F}\left(\frac{1}{n}\overline{F}^{-1}\left(\frac{\alpha}{n}\right)\right)\sim (\frac{1}{n})^{-\gamma}\overline{F}\left(\overline{F}^{-1}\left(\frac{\alpha}{n}\right)\right)=n^{\gamma}\frac{\alpha}{n}=\alpha n^{\gamma-1}$ from equation (\ref{eq:regvardef}). Since $\alpha n^{\gamma-1}\downarrow0$ when $0<\gamma<1$ and  $\alpha n^{\gamma-1}\uparrow\infty$ when $\gamma>1$, the sum-combination statistic cannot have asymptotically correct size unless $\gamma=1$ or $n=1$. If one is interested in keeping $\bm{P}(S_{1,\mathcal{S},\ref{assump:w_equal}(\bm{U}_n)}>t_{\alpha,n})\lesssim \alpha$, one needs $\gamma\in(0,1]$. This observation substantiates the validity of the sum-combination statistics with regularly varying inverse calibrator only when its index is 1 or less than 1, under the presence of a strong positive correlation.
}\end{example}

\begin{example}{\rm If $F$ is heavy-tailed Weibull with the survival function $\overline{F}(t)=e^{-t^k}$ for $0<k<1$ and $\left(\overline{F}^{-1}(y)\right)^{k}=-\log(y)$. Hence $\bm{P}(S_{1,\mathcal{S},\ref{assump:w_equal}}(\bm{U}_n)>t_{\alpha,n})=\overline{F}\left(\frac{1}{n}\overline{F}^{-1}(\alpha/n)\right)=\exp\left(-n^{-k}\{\overline{F}^{-1}(\alpha/n)\}^k\right)=\exp\left(n^{-k}\{\log(\alpha)-\log(n)\}\right)$. With some calculations, it can be shown that 
$\bm{P}(S_{1,\mathcal{S},\ref{assump:w_equal}}(\bm{U}_n)>t_{\alpha,n})=\alpha$ only when 
\begin{equation*}\label{eq:weibull:perfectcor}\alpha=\alpha(n,k)=\exp\left(\frac{\log (n)}{1-n^k}\right).\end{equation*}
If one is interested in the size control,
$\bm{P}(S_{1,\mathcal{S},\ref{assump:w_equal}}(\bm{U}_n)>t_{\alpha,n})\leq\alpha$ for
\begin{equation*}\label{eq:weibull:perfectcor}\alpha\geq\exp\left(\frac{\log (n)}{1-n^k}\right).\end{equation*}
This condition is equivalently written in terms of $k$ as
\begin{equation}\label{eq:weibull:perfectcor2}k= k(n,\alpha)\leq\frac{\log\left(1-\log(n)/\log(\alpha)\right)}{\log(n)}.\end{equation}
Since $\alpha(n,k)$ is an increasing function both in $n$ and $k$ and always takes values between 0 and 1, there will always be a unique value of $k$, if exists, that satisfies (\ref{eq:weibull:perfectcor2}) for given $\alpha$ and $n$.
A Weibull distribution has a lighter tail as $k\uparrow 1$ and heavier tail as $k\downarrow0$.
Since $k(n,\alpha)$ is decreasing in $n$ and increasing in $\alpha$, the criteria in (\ref{eq:weibull:perfectcor2}) states that the more perfectly, or strong-positively correlated, tests we have, or the smaller significance level $\alpha$ we require, the heavier tail of a Weibull calibrator we need to employ, but not too heavy.
}\end{example}

\begin{example}\label{example:logregvarying2}{\rm
When $F$ is log-Pareto with index $\gamma$, $\overline{F}(t)=\bm{P}(\log(X)>\log(t))=\{\log(t)\}^{-\gamma}$ and $\overline{F}^{-1}(t)=\exp(\alpha^{-\frac{1}{\gamma}}n^{\frac{1}{\gamma}})$, and hence $\overline{F}\left(\frac{1}{n}\overline{F}^{-1}\left(\frac{\alpha}{n}\right)\right)=\left[\log\left\{n^{-1}\exp\left(\alpha^{-\frac{1}{\gamma}}n^{\frac{1}{\gamma}}\right)\right\}\right]^{-\gamma}=\left(\alpha^{-\frac{1}{\gamma}}n^{\frac{1}{\gamma}}-\log n  \right)^{-\gamma}$.
Therefore,
 $\bm{P}(S_{1,\mathcal{S},\ref{assump:w_equal}}(\bm{U}_n)>t_{\alpha,n})\leq\alpha$ when
$\alpha\leq \left(\frac{n^{\frac{1}{\gamma}}-1}{\log n}\right)^{\gamma},$ and
the precise size control
$\bm{P}(S_{1,\mathcal{S},\ref{assump:w_equal}}(\bm{U}_n)>t_{\alpha,n})=\alpha$ can be achieved only if
\begin{equation}\label{eq:example:gamma:logpareto}\alpha=\alpha(n,\gamma)=\left(\frac{n^{\frac{1}{\gamma}}-1}{\log n}\right)^{\gamma}.\end{equation}
By taking derivatives of $\alpha(n,\gamma)$ with respect to $n$ and $\gamma$, respectively, we can see that $\alpha(n,\gamma)$ is increasing in $n$ and decreasing in $\gamma$. While there is no closed-form solution to find such $\gamma$, it is easy to see that the smaller $\alpha$ is or the larger $n$ is, the larger $\gamma$ is required. Since larger $\gamma$ indicates a lighter tail, we make the opposite conclusion to the Weibull case above - a lighter tail, but not too light, is required for a smaller significance level or for a greater number of tests,  when a log-regularly varying calibrator to combine tests that have a perfect or strong positive correlation.
}\end{example}

Based on the above examples, it seems that, within the same calibrator class and under a perfect or strong positive correlation, there is a sweet spot in choosing the indices related to how heavy the calibrator's right tail is.  It is interesting to note that the ideal tail index for subexponential calibrators, well beyond the regularly varying class, depends both on $n$ and $\alpha$. Table \ref{table:optimal_weibull_logP} presents such optimal parameters for Weibull and log-Pareto calibrators: the Weibull $k=\frac{\log\left(1-\log(n)/\log(\alpha)\right)}{\log(n)}$, and the log-Pareto $\gamma$ is chosen numerically on a fine grid $\gamma=0.001,0.002,0.003,\ldots,10$ so that the distance between $\alpha$ and $\left((n^{\frac{1}{\gamma}}-1)/ \log n \right)^{\gamma}$ is minimized for each combination of $\alpha$ and $n$.

\begin{table}[htbp]
    \centering
    \begin{tabular}{|c||ccccc|ccccc|}
        \hline
        \multirow{2}{*}{\diagbox{$n$}{$\alpha$}}& \multicolumn{5}{c|}{$k$ for Weibull} & \multicolumn{5}{c|}{$\gamma$ for log-Pareto} \\
        \cline{2-11}
        & 0.1 & 0.05 & 0.01 & 0.001 & 0.0001 & 0.1 & 0.05 & 0.01 & 0.001 & 0.0001\\
        \hline\hline
        25   & 0.2717 & 0.2267 & 0.1647 & 0.1188 & 0.0931 & 3.346 & 3.649 & 4.317 & 5.210 & 6.050 \\
        50   & 0.2538 & 0.2136 & 0.1572 & 0.1147 & 0.0905 & 3.524 & 3.819 & 4.474 & 5.354 & 6.185 \\
        100  & 0.2386 & 0.2022 & 0.1505 & 0.1109 & 0.0880 & 3.700 & 3.988 & 4.631 & 5.498 & 6.320 \\
        500  & 0.2105 & 0.1807 & 0.1374 & 0.1033 & 0.0830 & 4.104 & 4.378 & 4.994 & 5.834 & 6.636 \\
        1000 & 0.2007 & 0.1731 & 0.1326 & 0.1003 & 0.0810 & 4.275 & 4.544 & 5.150 & 5.978 & 6.773 \\
        \hline
    \end{tabular}\caption{The optimal shape parameters $k$ of Weibull calibrators and the optimal indexes $\gamma$ of log-Pareto calibrators that satisfy (\ref{eq:weibull:perfectcor2})  and (\ref{eq:example:gamma:logpareto}), respectively, for various choices of the significance level $\alpha$ and the number $n$ of tests. The Weibull $k$ is chosen as the upper limit of (\ref{eq:weibull:perfectcor2}), and the log-Pareto $\gamma$ is found via grid search, incrementally from 0.001 to 10 with a step size of 0.001.}\label{table:optimal_weibull_logP}
\end{table}

When $F$ is relatively light-tailed such as Weibull, it seems that the heavier tails help achieve robust results as $n$ increases or $\alpha$ decreases, whereas a relatively heavy-tailed calibrator such as log-regularly varying requires a lighter tail. Such tail index for the regularly varying class, in contrast, does not depend on $n$ nor $\alpha$. When $F\in\mathcal{RV}_{-\gamma}$, $\gamma=1$ is preferred for its robustness to the strong positive correlation among the tests. This easier choice of the tail index makes the regularly varying class more favorable when constructing a $p$-value combination strategy.
The next section explores the properties of regularly varying $F$.


\section{Regularly varying distributions and M-merging functions}\label{sec:regvar_Mmergfun}

This section explores the combination statistics with regularly varying $F$ when the number of tests $n$ is finite. We limit our interest to the regularly varying $F$ for three reasons. First, as can be seen from Table \ref{table:literature}, if a weaker dependence condition such as the UTAI or TAI is necessary, the calibrator should be limited to dominatedly varying or consistently varying. Regularly varying distributions $\mathcal{RV}_{-\gamma}$ with $\gamma>0$ is an important class in $\mathcal{C}$ and $\mathcal{D}\cap\mathcal{L}$. Second, by limiting the calibrator class to regularly varying, we can rewrite our sum- and max- combination statistics so that they fit into the {\it M-family merging function} framework of \cite{vovk2020combining}. Last but not least, the ability to control the tail probability provides a convenient way to build a critical value for the combined tests, as mentioned briefly after Theorem \ref{thm:inS}. Its associated parameters (in the regularly varying case, the index $\gamma$) that are robust to strong positive correlation are easy to find and do not necessarily have to depend on other parameters such as $n$, as Example \ref{example:regularlyvarying2} demonstrated.

We write \cite{chen2009sums}'s result below. Note that \cite{geluk2009asymptotic}, \cite{li2013pairwise}, or \cite{cang2017extremal}'s results can also be used for wider classes of calibrators but with more restrictive dependence assumptions.
\begin{theorem}[Theorem 3.2 of \cite{chen2009sums}]\label{thm:chen}
If Assumptions \ref{assump:w_random} and \ref{assmp:UTAI} are satisfied, the equation (\ref{eq:indep_tailprob}) holds for $j=1$, for any finite $n$, as long as $F\in \mathcal{C}$ with positive support.
\end{theorem}

Now we demonstrate that $S_{1,F,\ref{assump:w_equal}}(\bm{U}_n)$ and $S_{3,\mathcal{S},\ref{assump:w_equal}}(\bm{U}_n)$ with Pareto $F$  are equivalent to the M-merging functions in \cite{vovk2020combining,vovk2021values,vovk2022admissible}. It turns out that some definitions in \cite{vovk2020combining,vovk2021values,vovk2022admissible}'s framework need to be relaxed in order to accommodate other regularly varying $F$.
We also note that, unlike the sum combination statistic, the max-cumulative-sum and the max combination statistics with regularly varying $F$ are not yet covered in the $p$-value combination literature, to the best of our knowledge. We propose to utilize the tail equivalence literature to extend the type of $p$-value combination statistics as well as to consider random weights. Our argument below relies on the item \ref{thm:inS-2} of Theorem \ref{thm:inS} for $j=2$ and $j=3$ cases. While this setup limits the scope of the dependence structure, compared to the ones \cite{vovk2020combining}, we cover regularly varying calibrators and random weights.

The asymptotic scale invariance in equation  (\ref{eq:regvardef}) can be rewritten as, for any $a>0$,
\begin{equation}\label{eq:regvarying}\overline{F}(at)=\bm{P}(X>at)=\bm{P}(a^{-1}X>t)\sim a^{-\gamma}{\bm{P}(X>t)}=a^{-\gamma}\overline{F}(t),\end{equation}
as $t\to\infty$.
We also note that a similar relationship works when $a$ is a nonnegative random variable and is independent of $X$;
\begin{equation}\label{eq:regvarying_randomw}\overline{F}(at)\sim \left(\bm{E}a^{-\gamma}\right)\overline{F}(t),\end{equation}
as $t\to\infty$.
See, for instance, the proof of Theorem 2.2 of \cite{wang2006tail}, which was originally proven by \cite{breiman1965some}. Note that the random $a$ in (\ref{eq:regvarying_randomw}) obviously covers the nonrandom case in (\ref{eq:regvarying}).
Define
\begin{equation}\label{a_n}a_{n,\gamma}=\left(\sum_{i=1}^n \bm{E}w_i^{\gamma}\right)^{-\frac{1}{\gamma}}.\end{equation}
The scale invariance of regularly varying tails in (\ref{eq:regvarying_randomw}) implies
\begin{equation}\label{eq:regvarying1}\sum_{i=1}^n\bm{P}(w_iX_i>t)\sim \left(\sum_{i=1}^n \bm{E}w_i^\gamma\right)\overline{F}(t)=\left(\sum_{i=1}^n \bm{E}w_i^{\gamma}\right)\bm{P}(X>t)\end{equation}
and
\begin{equation}\label{eq:regvarying2} \left(\sum_{i=1}^n \bm{E}w_i^{\gamma}\right)\bm{P}(X>t)\sim\bm{P}(X>a_{n,\gamma}t).\end{equation}
The above approximation holds as long as all $X_i$ have the same distribution $F$ that has regularly varying tails and $t$ is large enough.
If the assumptions of Theorem \ref{thm:inS} or \ref{thm:chen} are satisfied, we have
  $\bm{P}(S_{j,F,w}(\bm{U}_n)>t)\sim
    \bm{P}\left(X>a_{n,\gamma}t\right)$
as $t\to\infty$.
Assuming that $a_{n,\gamma}t$ is large enough, by letting $t'=a_{n,\gamma}t$ and by writing $t'$ as $t$,
\begin{equation}\label{eq:comb.asymp}
    \bm{P}(a_{n,\gamma}S_{j,F,w}(\bm{U}_n)>t)\sim
    \bm{P}\left(X>t\right)
\end{equation}
as $t\to\infty$.
In fact, similar observations were made for $S_{1,F,w}(\bm{U}_n)$ by \cite{wilson2019harmonic} in equation (5) and by \cite{chen2023trade} in  Remark 2, but their discussions were limited to the Pareto distribution with shape parameter 1 or a Cauchy distribution.
We emphasize that this relation (\ref{eq:comb.asymp}) for $j=1$ holds as long as  $X_i$ are UTAI with $F\in\mathcal{RV}_{-\gamma}$ for $\gamma>0$, and for $j=1,2,3$ holds when $X_i$ satisfy Assumption  \ref{assmp:geluktang} or  \ref{assmp:kotang} with $F\in\mathcal{F}$ along with other minor assumptions.  

We can also rewrite (\ref{eq:comb.asymp}) in the form of p-merging functions in \cite{vovk2020combining}'s framework. By letting $t=F^{-1}(1-\epsilon)=\overline{F}^{-1}(\epsilon)$ for any  small enough $\epsilon>0$ such that $t$ and $a_{n,\gamma}t$ are large enough, as $\epsilon\to0$,
\begin{equation}\label{eq:tailprecisecombfun}
    \bm{P}\left[\overline{F}(a_{n,\gamma}S_{j,F,w}(\bm{U}_n))<\epsilon\right]\sim    \epsilon.
\end{equation}

This observation motivates {\it tail-precise combining functions} below. Our definition is motivated by the p-merging function framework of \cite{vovk2020combining,vovk2021values,vovk2022admissible}. A random variable $U$ is a {\it p-variable} if $P(U\leq\epsilon)\leq\epsilon$ for all $\epsilon\in(0,1)$, and an increasing Borel function $T:~[0,\infty)^n\to[0,\infty)$ is a {\it p-merging function} if  $T(U_1,\ldots,U_n)$ is a p-variable whenever $U_1,\ldots,U_n$ are all p-variables; that is, $\sup_{\bm{U}_n\in \mathcal{U}^n}{P}\left(T(\bm{U}_n\right) \leq\epsilon)\leq\epsilon$, letting $\mathcal{U}^n$ be the collection of p-merging functions based on $n$ p-variables. A p-merging function is {\it precise} if, for all $\epsilon\in(0,1)$, $\sup_{\bm{U}_n\in \mathcal{U}^n}{P}\left(T(\bm{U}_n\right) \leq \epsilon)=\epsilon.$

The notion of p-merging function is quite strong. For instance, \cite{ling2022stable} showed that for all $\epsilon\in(0,1)$,
$\sup_{\bm{U}_n\in\mathcal{U}_{mixing}^n}\lim_{n\to\infty}\bm{P}(T(\bm{U}_n)\leq\epsilon)=\epsilon$,
where $T(\bm{U}_n)=\overline{F}(a_{n,\gamma}S_{1,F,w}(\bm{U}_n))$ and ${F}$ is the distribution function of a stable random variable with stability parameter $\gamma$. Here, $\mathcal{U}_{ mixing}^n\subset\mathcal{U}^n$ is a set of vectors $\bm{U}_n$ of uniform random variables  that are weakly dependent, satisfying  mixing conditions and a short-range tail independence condition presented in \cite{ling2022stable}. In general, the random behavior of $T(\bm{U}_n)$ may depend on the strength and form of the dependence among $\bm{U}_n$.
The following definitions are to help distinguish qualities that affect the ability to control sizes.
\begin{definition}{\rm
Given a subset $\mathcal{U}_{\mathcal{A}}^n$ of $\mathcal{U}^n$,  we say a measurable function $T:[0,1]^n\to[0,\infty)$ is a {\it precise combining function for $\mathcal{U}_{\mathcal{A}}^n$} if, for all $\epsilon\in(0,1)$,
$$\sup_{\bm{U}_n\in\mathcal{U}_{\mathcal{A}}^n}\bm{P}\left(T(\bm{U}_n\right) \leq\epsilon)=\epsilon$$
and is an {\it asymptotically precise combining function for $\mathcal{U}_{\mathcal{A}}^\infty:=\{\mathcal{U}_{\mathcal{A}}^n\}_{n=1}^\infty$} if, for all $\epsilon\in(0,1)$,
$$\lim_{n\to\infty}\sup_{\bm{U}_n\in\mathcal{U}_{\mathcal{A}}^n}\bm{P}\left(T(\bm{U}_n\right) \leq\epsilon)=\epsilon.$$
We further define that $T(\bm{U}_n)$ is a  {\it tail-precise combining function for $\mathcal{U}_{\mathcal{A}}^n$} if, as $\epsilon\downarrow0$, $$\sup_{\bm{U}_n\in\mathcal{U}_{\mathcal{A}}^n}\bm{P}\left(T(\bm{U}_n\right) \leq\epsilon)\sim\epsilon.$$
Similarly, with $\epsilon=\epsilon_n\downarrow 0$ as $n\to\infty$, we say $T(\bm{U}_\infty)$ is an  {\it asymptotically tail-precise combining function for $\mathcal{U}_{\mathcal{A}}^\infty$} if, as $n\to\infty$, $$\sup_{\bm{U}_n\in\mathcal{U}_{\mathcal{A}}^n}\bm{P}\left(T(\bm{U}_n\right) \leq\epsilon)\sim\epsilon.$$

}\end{definition}

We use the term  ``combining function," than  ``p-merging function" to emphasize that our definition cannot be naively compared to p-merging functions. For instance, the stable combination statistics \citep{ling2022stable} are asymptotically precise combining functions for $\mathcal{U}_{mixing}^n$ but are not p-merging functions. We now formally state our results.
Recall that, under the global null, all $p_i=U_i$ and $\bm{X}_n=\bm{X}_n(\bm{U}_n)=\left(\overline{F}^{-1}(U_1),\ldots,\overline{F}^{-1}(U_n)\right)'$, and the combination statistics $S_{j,F,w}(\bm{U}_n)$ are functions of $F$, $\bm{U}_n,$ and $\{w_i\}$.  
Define
\begin{equation}\label{Tn}
    T_{j,
    F,w}(\bm{U}_n)=\overline{F}\left(a_{n,\gamma}S_{j,F,w}(\bm{U}_n)\right), ~~~~j=1,2,3.
\end{equation}
The following theorem organizes the above argument.
\begin{theorem}
Let $F\in \mathcal{RV}_{-\gamma}$ with $\gamma>0$, the weight $(w_1,\ldots,w_n)$ satisfies Assumption \ref{assump:w_random}, and the number $n$ of tests is finite.
The sum-combination statistic $T_{1,F,w}(\bm{U}_n)$ is a tail-precise combining function for $\mathcal{U}_{\ref{assmp:UTAI}}^n$, where $$\mathcal{U}_{Bj}^n=\{\bm{U}_n\in \mathcal{U}^n:~\bm{U}_n~{\rm satisfies~Assumption~}Bj,~U_i~{\rm uniform~on}~[0,1]~{\rm for~all~}i=1,\ldots,n\}.$$
Further, if $F$ satisfies Assumption \ref{assump:leftail}, all three combination statistics $T_{1,F,w}(\bm{U}_n)$, $T_{2,F,w}(\bm{U}_n)$, and $T_{3,F,w}(\bm{U}_n)$ are  tail-precise combining functions for $\mathcal{U}_{\ref{assmp:kotang}}^n$.

\end{theorem}


The M-family merging functions $\tilde{a}_{r,n}M_{r,n}$ in \cite{vovk2020combining,vovk2022admissible} can be compared to our tail-precise combining functions $T_{j,F,w}(\bm{U}_n)$ for all $j=1,2,3$, 
when the calibrator is an inverse Pareto survival function. Let $F$ be a Pareto distribution function with $\overline{F}(x)=x^{-\gamma}$ and $\overline{F}^{-1}(x)=x^{-\frac{1}{\gamma}}$ for any $\gamma>0$. Since $X_i$ are always positive, it is trivial that $T_{2,F,w}(\bm{U}_n)=T_{1,F,w}(\bm{U}_n)$.
Assume equal weights $w_i=1/n$. Noting that $a_{n,\gamma}=n^{(\gamma-1)/\gamma}$ when all $w_i=1/n$,  $T_{3,F,\ref{assump:w_equal}}(\bm{U}_n)=\overline{F}(a_{n,\gamma}S_{3,F,\ref{assump:w_equal}}(\bm{U}_n))=\left\{a_{n,\gamma}S_{3,F,\ref{assump:w_equal}}(\bm{U}_n)\right\}^{-\gamma}=\left\{n^{-\frac{1}{\gamma}}\max \left(U_i^{-\frac{1}{\gamma}}\right)\right\}^{-\gamma}=n\min_{i=1}^nU_i$, which is equivalent to the merging function $nM_{-\infty,n}$ of the M-family and is free of $\gamma$. \cite{vovk2020combining} pointed out that
$nM_{-\infty,n}$ is equivalent to the Bonferroni's method. Remark \ref{remark:maxcomstat} reconfirms such observation that $T_{3,F,\ref{assump:w_equal}}(\bm{U}_n)$ is indeed equivalent to Bonferroni's method for any 
 subexponential $F$.

Now we check the  relationship between $T_{1,F,\ref{assump:w_equal}}(\bm{U}_n)$ and the M-family merging functions $\tilde{a}_{r,n}M_{r,n}$. Since
 $S_{1,F,w}(\bm{U}_n)=\frac{1}{n}(U_1^{-\frac{1}{\gamma}}+\cdots+U_n^{-\frac{1}{\gamma}})=\left(M_{-\frac{1}{\gamma},n}\right)^{-\frac{1}{\gamma}},$
where $M_{r,n}(U_1,\ldots,U_n)=\{(U_1^r+\cdots+U_n^r)/n\}^{1/r}$ as defined in \cite{vovk2020combining},
\begin{equation}\label{eq:comb.fun.1}T_{1,F,\ref{assump:w_equal}}(\bm{U}_n)=\overline{F}(a_{n,\gamma}S_{1,F,\ref{assump:w_equal}}(\bm{U}_n))=a_{n,\gamma}^{-\gamma}S_{1,F,\ref{assump:w_equal}}(\bm{U}_n)^{-\gamma}=n^{1-\gamma}M_{-\frac{1}{\gamma},n}.\end{equation}
Table 1 of \cite{vovk2020combining} illustrates (asymptotically) precise merging functions $\tilde{a}_{r,n}M_{r,n}$ of M-family for different values of their $r$. Letting $r=-\frac{1}{\gamma}$, we can find the relationship between the M-family mergning functions and our tail precise combining functions $T_{j,F,\ref{assump:w_equal}}(\bm{U}_n)$. Since Pareto distribution requires $\gamma>0$, we present the cases with $r<0$ only.  \cite{vovk2020combining} noted that $\tilde{a}_{-1,n}M_{-1,n}$ is not a merging function  for finite $K$.

\bigskip
\begin{table}[h!]
\centering
\begin{tabular}{ c|c||c|c }
 Range of $r$ & Range of $\gamma$ & $\tilde{a}_{r,n}$ of M-family& $n^{1-\gamma}$ from $T_{1,F,\ref{assump:w_equal}}(\bm{U}_n)$\\ \hline 
 $r\in(-1,0)$ & $\gamma>1$ & $(1-\frac{1}{\gamma})^{-\gamma}$, asymptotically precise &$n^{1-\gamma}\downarrow0$ as $n\uparrow\infty$\\ 
 $r=-1$ & $\gamma=1$ & $\log n\uparrow\infty$, asymptotic formula
 &1 \\ 
 $r<-1$&$\gamma\in(0,1)$&$\frac{n^{1+\gamma}}{1-\gamma}\uparrow\infty$, asymptotically precise &$n^{1-\gamma}\uparrow\infty$ as $n\uparrow\infty$
 \\\hline\hline \multicolumn{2}{c||}{$r=-\infty$}&$\tilde{a}_{-\infty,n}M_{-\infty,n}=nM_{-\infty,n}$&$T_{3,F,\ref{assump:w_equal}}(\bm{U}_n)=nM_{-\infty,n}$
\end{tabular}
\caption{Comparisons between M-family merging functions $\tilde{a}_{r,n}M_{r,n}$ of \cite{vovk2020combining} and our merging functions, $T_{1,F,\ref{assump:w_equal}}(\bm{U}_n)=n^{1-\gamma}M_{-\frac{1}{\gamma},n}$ and $T_{3,F,\ref{assump:w_equal}}(\bm{U}_n)=nM_{-\infty,n}$, with Pareto distribution in the equal weights case.}\label{table:mfamily}
\end{table}

It is worth noting our asymptotically tail-precise combining function $T_{1,F,\ref{assump:w_equal}}(\bm{U}_n)$ tends to be much smaller than the M-merging function $\tilde{a}_{r,n}M_{r,n}$ for all $\gamma>0$.  Considering that our combining functions are asymptotically tail precise for small enough $\alpha$, we can expect that \cite{vovk2020combining}'s M-merging functions would tend to under-reject. 
Section \ref{sec:gammachoice} discusses how small $\alpha$ is required, including which $\gamma$ to choose when $\alpha$ is not small enough to guarantee the validity of our combining tests.
We also note that when $\gamma=1$ and equal weights, our $T_{1,F,\ref{assump:w_equal}}(\bm{U}_n)$ is equivalent to \cite{wilson2019harmonic}'s HMP without the Landau adjustment. 

Considering a diverging number $n$ of tests also provides some insights on the choice of $\gamma$. The next section explores this aspect.

\section{When the number of tests diverges
}\label{sec:ninft}
\cite{fang2024heavy} proved in their Theorem 3 that the power of their test reaches 1 as the number of tests $n$ goes to infinity. However, their test is shown to be controlled only for a fixed $n$. This section proves the validity of the $p$-value combination methods with regularly varying calibrators for diverging numbers of tests, which can be considered as an extension of \cite{fang2024heavy}'s size results.

We first introduce the sum- and cumsum-combination statistics cases, which have been quite thoroughly studied in the tail equivalence literature. See the second half of Table \ref{table:literature} for a brief summary of such literature for dependent $X_i$. In fact, for these two statistics, there are routine arguments that extend the finite $n$ results to infinite $n$. There are three conditions that these routine arguments mainly rely on: the equivalence of the tail probability for finite $n$,  some summability assumptions on the weights, and some assumptions on the calibrator's ability to control the tail probability. In particular, for the last condition, most of these arguments allow only up to $F\in \mathcal{ERV}$ or $F\in\mathcal{C}$, as can be seen in Table \ref{table:literature}. It seems that subexponential calibrators beyond $\mathcal{C}$ cannot be used to guarantee the size control for diverging $n$. This suggests that subexponential distributions with too light or too heavy tails, such as Weibull and log-regularly varying, better be avoided being used as calibrators, despite their ability to control the size in finite samples.

Although $\mathcal{C}$ or $\mathcal{ERV}$ are slightly wider than $\mathcal{RV}_{-\gamma}$, in this section, we will limit our analysis to regularly varying calibrators $F\in\mathcal{RV}_{-\gamma}$. For one thing, (\ref{eq:tailprecisecombfun}) does not hold without the asymptotic scale invariance. This means that terms such as ``tail-precise'' cannot be used if one wishes to slightly extend the calibrator class beyond the regularly varying tails. The consistently varying class $\mathcal{C}$ still does not include other important subexponential distributions anyway.

In the $n\to\infty$ case, equation (\ref{eq:indep_tailprob}) would have been written as
\begin{equation}\label{eq:thm_ninfty}
\bm{P}\left(S_{j,F,w}(\bm{U}_\infty)>t\right)\sim\sum_{i=1}^\infty\bm{P}(w_iX_i>t)
\end{equation}
as $t\to\infty$, where $S_{1,F,w}(\bm{U}_\infty)=\sum_{i=1}^\infty w_i X_i=\sum_{i=1}^\infty w_i \overline{F}^{-1}(U_i)$, $S_{2,F,w}(\bm{U}_\infty)=\bigvee_{n=1}^\infty\left(\sum_{i=1}^n w_i X_i\right)$, and $S_{3,F,w}(\bm{U}_\infty)=\bigvee_{i=1}^\infty w_iX_i$. However, for $j=1$ and $2$, this naive extension (\ref{eq:thm_ninfty}) to an infinite $n$ case does not work unless the critical value $t$ also diverges as $n\to\infty$ at a certain rate of $n$. This will be discussed after introducing three propositions below. Before we delve into our $n$ infinite case, we first define pseudo-weights $\{\widetilde{w}_i\}$ and some summability conditions of pseudo-weights to introduce existing literature. A pseudo-weight $\widetilde{w}_i$, in our context, will be a function of a raw weight $w_i$ and the number of tests $n$, where $w_i$ can be either random or nonrandom. 
These summability conditions depend on the magnitude of the index $\gamma$ of $\mathcal{RV}_{-\gamma}$.

When weights are nonrandom but possibly with unequal weights as in Assumption \ref{assump:w_nonrandom}, \cite{bae2017note}'s condition is useful, of which result is further extended by \cite{gao2019asymptotic}. Although the original conditions allow negative weights, we state their condition only for positive weights according to our setting.
\begin{assumpD}[\cite{bae2017note,gao2019asymptotic}]\label{assump:wi_ninfty_gao}
Consider a nonnegative nonrandom sequence $\{\tilde{w}_i\}$. There exists $0\leq q\leq \min\{1,\gamma\}$ such that
$$\sum_{i=1}^\infty \widetilde{w}_i^{q}<\infty.$$
\end{assumpD}

When weights are allowed to be random as in Assumption \ref{assump:w_random} but with nonnegative calibrated $p$-values, a little heavier assumption on the pseudo-weights is necessary.
\begin{assumpD}[\cite{wang2006tail,chen2009sums}]\label{assump:wi_ninfty_chen}
   Consider a nonnegative random sequence $\{\tilde{w}_i\}$. When $0<\gamma< 1$, there exist $q_1$ and $q_2$ such that $0<q_1<\gamma<q_2<1$ and
$$\sum_{i=1}^\infty \left\{\bm{E}\widetilde{w}_i^{q_1}\vee \bm{E}\widetilde{w}_i^{q_2}\right\}<\infty.$$
When $\gamma\geq 1$, there exist $q_1$ and $q_2$ such that $0<q_1<\gamma<q_2<\infty$ and
$$\sum_{i=1}^\infty\left\{(\bm{E}\widetilde{w}_i^{q_1})^{\frac{1}{q_1}}\vee (\bm{E}\widetilde{w}_i^{q_2})^{\frac{1}{q_2}}\right\}<\infty.$$
\end{assumpD}

In addition, if the calibrated $p$-values are possibly negative, a slightly stronger set of assumptions is needed.
\begin{assumpD}[\cite{yi2011approximation,geng2019tail}]\label{assump:wi_ninfty_yi}
 Consider a nonnegative random sequence $\{\tilde{w}_i\}$.  When $0<\gamma< 1$, there exists  $\delta\in(0,\min\{\gamma,1-\gamma\})$ such that
$$\sum_{i=1}^\infty \bm{E}\widetilde{w}_i^{\gamma-\delta}<\infty.$$ 
When $\gamma\geq 1$, there exists $\delta\in(0,\gamma)$ such that $$\sum_{i=1}^\infty (\bm{E}\widetilde{w}_i^{\gamma-\delta})^{\frac{1}{\gamma+\delta}}<\infty.$$ 
\end{assumpD}

The following three propositions are due to Theorem 2.1 of \cite{gao2019asymptotic}, Theorem 3.3 of \cite{chen2009sums}, and Theorem 3.2 of \cite{geng2019tail}, respectively. We simplify their statements to limit the calibrators to the regularly varying ones. Recall that for any regularly varying $F\in\mathcal{RV}_{-\gamma}$, $\sum_{i=1}^\infty\bm{P}(\widetilde{w}_iX_i>t)\sim \left(\sum_{i=1}^\infty \widetilde{w}_i^\gamma\right)\overline{F}(t)$ when the weights are nonrandom, and  $\sum_{i=1}^\infty\bm{P}(\widetilde{w}_iX_i>t)\sim \left(\sum_{i=1}^\infty \bm{E}\widetilde{w}_i^\gamma\right)\overline{F}(t)$ when the weights are possibly random.

\begin{proposition}[\cite{gao2019asymptotic}]\label{proposition:infinite_n_gao}
Let  $F\in\mathcal{RV}_{-\gamma}$ with $p_F=1$ and $q_F=0$ from the balance condition. Let  $X_i=\overline{F}^{-1}(U_i)$. 
 Assume that $\{U_i\}_{i=1}^n$ satisfies Assumption \ref{assmp:geluktang} or  \ref{assmp:kotang} for each $n$.  If the pseudo weights $\{\widetilde{w}_i\}$ satisfy Assumption \ref{assump:wi_ninfty_gao}, as $t\to\infty$,
    $$\bm{P}\left(\sum_{i=1}^\infty \widetilde{w}_iX_i>t\right)\sim \left(\sum_{i=1}^\infty \widetilde{w}_i^\gamma\right)\overline{F}(t).$$
\end{proposition}

\begin{proposition}[\cite{chen2009sums}]\label{proposition:infinite_n_chen}
Let  $F\in\mathcal{RV}_{-\gamma}$. 
 Assume that the calibrated $p$-values $\{X_i\}_{i=1}^n$ satisfy Assumption \ref{assmp:UTAI} for each $n$, and $X_i=\overline{F}^{-1}(U_i)$ are nonnegative. If the nonnegative random pseudo-weights $\{\widetilde{w}_i\}$ are independent of the calibrated $p$-values $\{X_i\}$ and satisfy  Assumption \ref{assump:wi_ninfty_chen}, as $t\to\infty$,
    $$\bm{P}\left(\sum_{i=1}^\infty \widetilde{w}_iX_i>t\right)\sim \left(\sum_{i=1}^\infty \bm{E}\widetilde{w}_i^\gamma\right)\overline{F}(t).$$
\end{proposition}

\begin{proposition}[\cite{geng2019tail}]\label{proposition:infinite_n_geng}
    Let  $F\in\mathcal{RV}_{-\gamma}$ that satisfies Assumption \ref{assump:leftail}. Let $X_i=\overline{F}^{-1}(U_i)$. 
 Assume that  the calibrated $p$-values $\{X_i\}_{i=1}^n$ satisfy Assumption \ref{assmp:kotang} for each $n$. If the nonnegative random pseudo-weights $\{\widetilde{w}_i\}$ satisfy Assumption \ref{assump:wi_ninfty_yi}, as $t\to\infty$,
    $$\bm{P}\left(\sum_{i=1}^\infty \widetilde{w}_iX_i>t\right)\sim\bm{P}\left(\bigvee_{n=1}^\infty\sum_{i=1}^n \widetilde{w}_iX_i>t\right)\sim \left(\sum_{i=1}^\infty \bm{E}\widetilde{w}_i^\gamma\right)\overline{F}(t).$$
\end{proposition}

Assumptions \ref{assump:wi_ninfty_gao}--\ref{assump:wi_ninfty_yi} are not satisfied  
by naively letting  $\widetilde{w}_i=w_i$ with $\sum_{i=1}^nw_i=1$ for each $n$. For instance, in the equal weights case  $w_i=n^{-1}$,
$\sum_{i=1}^\infty w_i^q= \sum_{i=1}^\infty n^{-q}=n^{1-q}\to\infty$ as $n\to\infty$ for any $q\in(0,1)$. This means that the critical value $t_{\alpha,n}=(p_Fl_F)^{\frac{1}{\gamma}}\alpha^{-\frac{1}{\gamma}}\left(\sum_{i=1}^n \bm{E}w_i^\gamma\right)^{\frac{1}{\gamma}}$ that makes $\sum_{i=1}^n\bm{P}(w_iX_i>t_{\alpha,n})\sim\alpha$ should diverge as $n\to\infty$. Therefore, unless the significance level $\alpha=\alpha_n$ coverage to 0 as $n\to\infty$, Assumptions \ref{assump:wi_ninfty_gao}--\ref{assump:wi_ninfty_yi} cannot be satisfied. 

Therefore, we need some additional assumptions on the significance level $\alpha_n$'s rate of convergence to 0. The following two assumptions on the significance level $\alpha_n$ are for nonrandom weights. We first start with a case when a lower bound on the minimum weight is assumed to be of order $n^{-1}$, 
\begin{equation}\label{assum:minweight}\min_{i=1,\ldots,n}w_i\geq c_0n^{-1}\end{equation}
for some positive constant $c_0$. In other words, $\min_{i=1,\ldots,n}w_i=O(n^{-1})$.  Note that (\ref{assum:minweight})  also includes the case of the equal weight, Assumption \ref{assump:w_equal}. This assumption on the minimum weight has also been assumed in \cite{liu2020cauchy}'s Theorem 3 for the $n\to\infty$ case. Under the assumption (\ref{assum:minweight}), the following condition on the significance level is needed to guarantee the tail equivalence of the sum-combination statistic with $n\to\infty$.
\begin{assumpE}\label{assump:alpha_equalw}
     The significance level $\alpha_n$ converges to zero as the number $n$ of tests increases such that 
  $\alpha_n=o(n^{-(\gamma-1)    \mathbb{I}_{\{\gamma>1\}}})$ for any given $\gamma>0$.
\end{assumpE}
 This assumption  requires $\alpha_n=o(1)$ for any $\gamma\in(0,1]$ regardless of the speed of the increase of $n$, whereas when $\gamma>1$, $\alpha_n$ should converge to 0 exponentially fast as $n$ increases: $\alpha_n=o(n^{-\gamma-1})$. Therefore, when $w_i$ are nonrandom and asymptotically equal, $\gamma\in(0,1]$ are preferred over $\gamma>1$ due to their ability to control the size with a relatively larger significance level.

In order to accommodate any weights including an extreme case with $w_i=1$ for some $i$ and $w_j=0$ for all $j\neq i$, one needs a stronger condition on the significance level $\alpha_n$ for $\gamma\in(0,1)$. 
\begin{assumpE}\label{assump:alpha_nonrandw}
     The significance level $\alpha_n$ converges to zero as the number $n$ of tests increases such that 
  $\alpha_n=o(n^{-|1-\gamma|})$ for any given $\gamma>0$.
\end{assumpE}

The following theorem, which is mainly due to Proposition \ref{proposition:infinite_n_gao}, states the validity of the sum-combination statistic  with nonrandom weights. regularly varying calibrators that is either nonnegative or with limited left tails are considered. The size of the global test for large $n$ can be controlled as long as the significance level $\alpha_n$ is small enough.  

\begin{theorem}\label{thm:ninfty_gao} Let $F\in\mathcal{RV}_{-\gamma}$ with some $\gamma>0$. Consider a calibrator $F$ and  the calibrated $p$-value $\{X_i\}_{i=1}^n$ that satisfy one of the followings:
 \begin{enumerate}[{\normalfont (F-i)}]
 \item $p_F=1$ and $q_F=0$ from the balance condition for $F$, and  $\{X_i\}_{i=1}^n$ satisfies Assumption $Bj$, with $j=3$ or 4 (Assumptions \ref{assmp:geluktang} or  \ref{assmp:kotang}) for each $n$.
 \item $F$ is concentrated on $[0,\infty)$, and $\{X_i\}_{i=1}^n$ satisfies Assumption $Bj$, with $j=1$ or 2 (Assumptions \ref{assmp:UTAI} or \ref{assmp:TAI})  for each $n$.
 \end{enumerate}
 In addition, assume that the weight $\{w_i\}_{i=1}^n$ and the significance level $\alpha=\alpha_n$ satisfy one of the followings:
 \begin{enumerate}[{\normalfont (w-i)}]
   \item $\{w_i\}_{i=1}^n$ satisfies Assumption \ref{assump:w_equal} for each $n$, and $\alpha_n$ satisfies Assumption \ref{assump:alpha_equalw}.
     \item $\{w_i\}_{i=1}^n$ satisfies Assumption \ref{assump:w_nonrandom} for each $n$, and $\alpha_n$ satisfies Assumption \ref{assump:alpha_nonrandw}.
 \end{enumerate}
 Then with the significance level $\alpha_n$ above, the sum-combination statistic $T_{1,F,w}(\bm{U}_n)$ defined in (\ref{Tn}) is an asymptotically tail-precise combining function for $\mathcal{U}_{Bj}^\infty$. 
  %
\end{theorem}

With random weights as in Assumption \ref{assump:w_random}, one needs an even stronger assumption on the significance level for $\gamma>1$.
\begin{assumpE}\label{assump:alpha_randw}
    The significance level $\alpha_n$ converges to zero as the number $n$ of tests increases such that 
  $\alpha_n=o(n^{-(1-\gamma)})$ when $0<\gamma\leq1$ and   $\alpha_n=o(n^{-2(\gamma-1)})$ when $\gamma>1$.
\end{assumpE}

The following theorem is mainly due to Proposition \ref{proposition:infinite_n_chen}. We can also allow random weights but need to assume nonnegative calibrated $p$-values.

\begin{theorem}\label{thm:ninfty_chen}
Let $F\in\mathcal{RV}_{-\gamma}$ with some $\gamma>0$. Consider the calibrated $p$-values $\{X_i\}_{i=1}^n$ satisfies Assumption \ref{assmp:UTAI} for each $n$ and are nonnegative. Assume that the weight $\{w_i\}_{i=1}^n$ satisfies Assumption \ref{assump:w_random} for each $n$.
If the significance level $\alpha_n$ satisfies Assumption \ref{assump:alpha_randw}, the sum-combination statistic $T_{1,F,w}(\bm{U}_n)$ 
 defined in (\ref{Tn}) is an asymptotically tail-precise combining function for $\mathcal{U}_{\ref{assmp:UTAI}}^\infty$.
\end{theorem}

If the calibrated $p$-values are not necessarily nonnegative, we need even stronger assumptions for the dependence structure to ensure the validity of the sum- as well as the cumsum-combination statistics. The following theorem is mainly based on Proposition \ref{proposition:infinite_n_geng}.

\begin{theorem}\label{thm:ninfty_geng}
Let $F\in\mathcal{RV}_{-\gamma}$ with some $\gamma>0$ satisfies Assumption \ref{assump:leftail}.
 Assume that, for each $n$, the calibrated $p$-value $\{X_i\}_{i=1}^n$ satisfies Assumption \ref{assmp:kotang} and the weight $\{w_i\}_{i=1}^n$ satisfies Assumption \ref{assump:w_random}.
If the significance level $\alpha_n$ satisfies Assumption \ref{assump:alpha_randw}, the sum- and cumsum-combination statistics $T_{1,F,w}(\bm{U}_n)$  and $T_{2,F,w}(\bm{U}_n)$,
 defined in (\ref{Tn}), are asymptotically tail-precise combining functions for $\mathcal{U}_{\ref{assmp:kotang}}^\infty$.
\end{theorem}

\begin{remark}{\rm
\cite{liu2020cauchy} proved a similar result for CCT. One of their conditions, (C.2), controls the maximum eigenvalue of the correlation matrix $\Sigma$ of $\bm{Y_n}$: $\lambda_{\max}(\Sigma)\leq C_0$ for a positive constant $C_0$, where $\lambda_{\max}(A)$ indicate the maximum eigenvalue of a matrix $A$. We note that this condition may be stronger than ours. For instance, if $\Sigma$ takes an equicorrelation structure, $\Sigma=(1-\rho)I_n+ \rho J_n$, where $I_n$ is the $n$ by $n$ identity matrix and $J_n$ is the $n$ by $n$ matrix of ones, the maximum eigenvalue is $(n-1)\rho+1$, which diverges to $\infty$ as $n\to\infty$. Therefore, (C.2) of \cite{liu2020cauchy} requires $\rho=O(1/n)$ or slower in the equicorrelation case. On the contrary, our Assumption \ref{assmp:UTAI} governs only the bivariate relationship, even in the diverging $n$ case.
}\end{remark}

We finish this section with a brief discussion on the max-combination test with a diverging number of tests.
When $X_i$ are iid, the approximate tail probability of the max-combination statistic under the null is well-known in the extreme value literature. A regularly varying distribution $F$ falls into the maximum domain of attraction of Fr\'echet distribution. That is, if the distribution function $F$ of $X_i$ is regularly varying with index $\gamma$,
\begin{equation}\label{eq:max_ninfty}
\lim_{n\to\infty}\bm{P}\left(c_n^{-1} \bigvee_{i=1}^n X_i<t\right)=\lim_{n\to\infty} \{F(c_nt)\}^n=\exp(-t^{-\gamma})
\end{equation}
for all real $t>0$, where $c_n=\overline{F}^{-1}(n^{-1})$. In fact, there is an if-and-only-if relationship between $F\in\mathcal{RV}_{-\gamma}$ and equation (\ref{eq:max_ninfty}). See, for instance, Proposition 1.11 of  \cite{resnick2008extreme}, which was first discovered by \cite{gnedenko1943distribution}. 
The following theorem uses (\ref{eq:max_ninfty}) to prove that max-combination statistics with regularly varying calibrators and with equal weights control the size well in the iid case. Due to the equivalence of the max-combination statistic and Bonferroni's method, this theorem can also be understood to prove the validity of Bonferroni's method with infinitely many tests as long as these tests are independent. 

\begin{theorem}\label{thm:max_iid_gnedenko}
    Let $F\in\mathcal{RV}_{-\gamma}$ with index $\gamma>0$. Let the sequence of $p$-values $U_1,\ldots,U_n$ be generated under the null and be independent. Assume equal weights, or Assumption \ref{assump:w_equal}. For a small enough significance level $\alpha>0$, 
\begin{equation}\label{thm:eq:max_ninfty}\bm{P}(S_{3,F,\ref{assump:w_equal}}(\bm{U}_n)>t_{\alpha,n})\sim \alpha~~~~{\rm as}~n\to\infty,\end{equation}
where $t_{\alpha,n}=\frac{1}{n}\overline{F}^{-1}(\alpha/n)$.
\end{theorem}

Indeed, the tail equivalence (\ref{eq:indep_tailprob}) holds for $S_{3,F,w}(\bm{U}_n)$, regardless of the choice of the calibrator, with infinitely growing number $n$ of tests, as long as these tests are independent. The nonequal weights that are possibly random are also allowed.
\begin{theorem}\label{thm:ninfty_max}
Let $F$ be any distribution function. Let the sequence of $p$-values $U_1,\ldots,U_n$ be generated under the null and be independent. Assume nonrandom or random weights, or Assumption \ref{assump:w_random}.
The tail equivalence works for the max-combination test for any $n$,
\begin{equation}\label{eq:thm:max_n}\bm{P}(S_{3,F,w}(\bm{U}_n)>t)\sim \sum_{i=1}^n\bm{P}(w_iX_i>t)~~~~{\rm as}~t\to\infty,\end{equation} 
implying that (\ref{eq:thm_ninfty}) holds for $j=3$ under independence.
In other words, the max-combination statistic $T_{3,F,w}(\bm{U}_n)$ with any calibrator and weights is a tail-precise combining function for $\mathcal{U}_{indep}^n$ for any $n$, and therefore, is an asymptotically tail-precise combining function for $\mathcal{U}_{indep}^\infty$. Here, $$\mathcal{U}_{indep}^n=\{\bm{U}_n\in\mathcal{U}^n:\bm{U}_n =(U_1,\ldots,U_n),~U_i\stackrel{iid}{\sim}Uniform[0,1]\}.$$
\end{theorem}

Theorem \ref{thm:ninfty_max} confirms that, in the independent case, a Bonferroni-like method can still control the size well even for infinitely growing $n$, as long as the significance level $\alpha$ reduces to 0 as $n\to\infty$.  
However, based on our simulation presented in Tables \ref{table:size} and \ref{table:rawpower}, this Bonferroni-like approach, the max-combination statistic, starts to fail when tests are dependent. As well-known in this area, Bonferroni-like approaches lose a significant amount of power when tests are strongly dependent, and therefore, not recommended in practice when strong dependency is suspected.

Now we are ready to discuss the choice of $\gamma$ for regularly varying calibrators.

\section{The choice of $\gamma$ and $\alpha$ from the size control perspective}\label{sec:gammachoice}

The validity of the combination statistics has been addressed only in an asymptotic sense with ``small enough" $\alpha$. But how small $\alpha$ should be in order to satisfy the tail approximations? For a given $\alpha$, what which value of $\gamma$ would be reasonable if we consider regularly varying calibrators? This section aims to address such questions from a theoretical standpoint.

At the end of Section \ref{sec:regvar_Mmergfun}, we briefly discussed that our $T_{1, F,w}(\bm{U}_n)=\overline{F}(a_{n,\gamma}S_{1,F,w}(\bm{U}_n))$ with  Pareto calibrator with index 1 and with equal weights is equivalent to \cite{wilson2019harmonic}'s HMP without the Landau adjustment. This adjustment is based on the Generalized Central Limit Theorem (GCLT) \citep{uchaikin2011chance} in an iid case. It is well-known in the literature that the Pareto calibrators with $\gamma>1$ are not as optimal. We also found the over-rejecting behavior when $\gamma>1$ in Section \ref{sec:subexponential} and in the $n\to\infty$ case in Section \ref{sec:ninft}, as seen from Assumption \ref{assump:alpha_equalw}. Therefore, we focus only on the $\gamma\in(0,1]$ for the rest of this section.

For a Pareto calibrator with index $\gamma>0$, $\overline{F}^{-1}(p)=p^{-\frac{1}{\gamma}}$ and with equal weights, our sum-combination statistic is $S_{1,F,w}(\bm{U}_n)=n^{-1}\sum_{i=1}^nU_i^{-\frac{1}{\gamma}}$. Following \cite{wilson2020generalized}'s argument, the GCLT states that if $U_i$ are iid, as $n\to\infty$,
\begin{equation}\label{eq:gcltgamma1}\frac{nS_{1,F_1,\ref{assump:w_equal}}(\bm{U}_n)-n\log n}{\pi n/2}\stackrel{d}{\longrightarrow} S(1,1,1,0;0),~~~~{\rm if}~\gamma=1,\end{equation}
and 
\begin{equation}\label{eq:gcltgamma2}\frac{nS_{1,F_\gamma,\ref{assump:w_equal}}(\bm{U}_n)}{n^{\frac{1}{\gamma}}\left\{\frac{2}{\pi}\Gamma(\gamma)\sin\left(\frac{\pi\gamma}{2}\right)\right\}^{-\frac{1}{\gamma}}}\stackrel{d}{\longrightarrow} S\left(\gamma,1,1,\tan\left(\frac{\pi\gamma}{2}\right);0\right),~~~~{\rm if}~\gamma\in(0,1).\end{equation}
Here, $S(\gamma,\beta,\sigma,\mu;0)$ indicates a stable distribution, following \cite{nolan2020univariate}'s S0 parametrization,  with stability index $\gamma$, skewness $\beta$, scale $\sigma$, and location $\mu$.

The following theorem states the approximate tail probability of the sum-combining statistic when individual tests are iid.

\begin{theorem}\label{thm:gclt}
Let $F\in\mathcal{RV}_{-\gamma}$ with some $\gamma\in(0,1]$ and the  calibrated $p$-value $\{X_i\}_{i=1}^n$ be iid. Assume the equal weights Assumption \ref{assump:w_nonrandom} and that the significance level $\alpha_n$ converges to zero at the following rates: $\alpha_n=o(1/\log n)$ if $\gamma=1$ and $\alpha_n=o(1)$ if $\gamma\in(0,1)$. Then the tail probability of the sum-combining function can be approximated as
$$\bm{P}\left(T_{1,F_\gamma,\ref{assump:w_equal}}(\bm{U}_n)<\alpha_n\right)\sim\left\{\begin{array}{ll}(\alpha_n^{-1}-\log n)^{-1}\sim \alpha_n,&\gamma=1
\\ \alpha_n,&\gamma\in(0,1),\end{array}\right.$$
as $n\to\infty$.
\end{theorem}

Section 1B in SI Appendix of \cite{wilson2019harmonic} provides another way to approximate the tail probability of the sum-combining function when $U_i$ are iid and $\gamma=1$ via Landau distribution approximation: \begin{equation}\label{eq:landau}\bm{P}(T_{1, F_\gamma,\ref{assump:w_equal}}(\bm{U}_n)<\alpha_n)\sim \bm{P}\left(S(1,1,\pi/2,\log n+0.874367)>\alpha_n^{-1}\right)\sim \alpha_n\end{equation}
as $\alpha_n\downarrow0$, where $S(\gamma,\beta,\sigma,\mu)$ indicates a random variable that follows $S(\gamma,\beta,\sigma,\mu;0)$. The last approximation is due to Theorem 1.2 of \cite{nolan2020univariate}.
Combining Theorem \ref{thm:gclt} and (\ref{eq:landau}), we can see that  $T_{1,F,\ref{assump:w_equal}}(\bm{U}_\infty )$ is an {\it asymptotically tail-precise combining function for} $\mathcal{U}_{indep}^\infty$, where $\mathcal{U}_{indep}^n$ as defined in Theorem \ref{thm:ninfty_max}.
This asymptotic tail-preciseness is guaranteed 
if $\gamma\in(0,1)$ and $\alpha_n=o(1)$ or if $\gamma=1$ and $1/\alpha_n\gg\log n$.
In particular, the latter condition for the $\gamma=1$ case is much stronger than that required for dependent $U_i$ as in  Assumption \ref{assump:alpha_equalw}, which reduces to  $\alpha_n=o(1)$ with  $\gamma\in(0,1]$.

Notice that (\ref{eq:landau}) suggests that there are two layers of approximations. We first investigate the second approximation, the tail probability approximation of a Landau distribution.

\begin{table}[htbp]
\centering
\begin{tabular}{|r||rrrrrr|}
  \hline
\diagbox{$n$~~~~~~~~}{$\alpha$} & 0.1 & 0.05 & 0.01 & 0.001 & 0.0001 & 0.00001 \\ 
  \hline\hline
10 & 1.6082 & 1.3211 & 1.0731 & 1.0093 & 1.0012 & 1.0001 \\
  25 & 1.8281 & 1.3997 & 1.0836 & 1.0102 & 1.0012 & 1.0001 \\ 
  50 & 2.0327 & 1.4653 & 1.0917 & 1.0109 & 1.0013 & 1.0002 \\ 
  100 & 2.2807 & 1.5371 & 1.0999 & 1.0116 & 1.0014 & 1.0002 \\ 
  1,000 & 3.6178 & 1.8323 & 1.1282 & 1.0140 & 1.0016 & 1.0002 \\ 
  10,000 & 6.4902 & 2.2563 & 1.1579 & 1.0164 & 1.0018 & 1.0002 \\ 
  100,000 & 9.8287 & 2.9077 & 1.1893 & 1.0187 & 1.0021 & 1.0002 \\ 
  1,000,000 & 10.0000 & 4.0037 & 1.2223 & 1.0211 & 1.0023 & 1.0003 \\ 
  10,000,000 & 10.0000 & 6.0846 & 1.2573 & 1.0235 & 1.0025 & 1.0003 \\ 
  100,000,000 & 10.0000 & 10.5354 & 1.2942 & 1.0260 & 1.0028 & 1.0003 \\ 
  1,000,000,000 & 10.0000 & 18.1519 & 1.3334 & 1.0284 & 1.0030 & 1.0003 \\ 
  10,000,000,000 & 10.0000 & 20.0000 & 1.3751 & 1.0308 & 1.0032 & 1.0003 \\ 
   \hline
\end{tabular}\caption{The ratio between tail probability  $\bm{P}(S(1,1,\pi/2,\log n+0.874367)>\alpha^{-1})$ and $\alpha$ for various choices of the significance level $\alpha$ and the number $n$ of tests.}\label{table:landau}
\end{table}

Table \ref{table:landau} reports the ratio  $$\frac{\bm{P}\left(S(1,1,\pi/2,\log n+0.874367)>\alpha^{-1}\right)}{\alpha}$$ with various choices of the significance level $\alpha$ and the number of tests $n$. These values are computed using the \texttt{pEstable} function of the \texttt{FMStable} package in \texttt{R}, as suggested by \cite[SI Appendix, page 3]{wilson2019harmonic}.  If this ratio is close to 1, this implies that the second approximation in (\ref{eq:landau}), the tail approximation of a Landau distribution, is accurate.  As this table suggests, the tail probability of a Landau random variable, compared $\alpha^{-1}$, is always greater than $\alpha$, which leads to over-rejection. This over-rejection becomes more serious as $n$ or $\alpha$ increases. While the severity of over-rejection varies, unless the number of observations is extremely large, $\alpha\leq0.01$ generally produces reasonably accurate approximation. On the contrary, $\alpha\geq 0.05$ is not accurate in general, and the size distortion becomes drastically severe as $n$ increases. To avoid this issue, \cite{wilson2019harmonic} suggested replacing the significance threshold with $\alpha_{n,wilson}$. Writing their suggestion in parallel to (\ref{eq:landau}) would lead to  
\begin{equation}\label{eq:landau-wilson}\bm{P}(T_{1, F_\gamma,\ref{assump:w_equal}}(\bm{U}_n)<\alpha_{n,wilson})\sim \bm{P}\left(S(1,1,\pi/2,\log n+0.874367)>\alpha_{n,wilson}^{-1}\right)\sim \alpha_n,\end{equation}
 where $\alpha_{n,wilson}$ are from Table 1 of \cite[p.1197]{wilson2019harmonic} and $\alpha_n$ is the nominal significance level. 
 This statement is accurate when $U_i$ are iid. However, when $U_i$ is not independent, the first approximation is no longer valid. In fact, we already know that 
 \begin{equation}\label{eq:T1tailprob}\bm{P}(T_{1, F_\gamma,\ref{assump:w_equal}}(\bm{U}_n)<\alpha_{n})\sim\alpha_n\end{equation}
as $\alpha_n\downarrow0$. Since $\alpha_{n,wilson}<\alpha_{n}$, the Wilson's HMP with the Landau adjustment would generally lead to under-rejection in finite samples. This under-rejection typically leads to the loss of power, as can be seen in Table \ref{table:rawpower}. In other words, when the approximation  (\ref{eq:T1tailprob}) works well, i.e., when  $\alpha$ is small enough, the sum-combination test with Pareto calibrator with index $\gamma=1$ without the Landau adjustment works better than the one with it. On the contrary, as Table \ref{table:size} suggests, when $\alpha$ is not small enough for this tail approximation (\ref{eq:T1tailprob}), Pareto calibrator with $\gamma=1$ without the Landau adjustment has severe size distortions, particularly for large $n$. The Landau adjustment or a Pareto calibrator with index $\gamma<1$ does to alleviate the severe over-rejections. 
This observation leads to the following guidance in the choice of $\gamma$. 
\begin{enumerate}
\item If the significance level $\alpha$ is small enough that $1/\alpha$ dominates $\log n$, a Pareto calibrator with $\gamma=1$ and without the Landau adjustment can be used without concerns for severe over-rejection. In general, $\alpha=0.01$ or smaller generally works okay. See Table \ref{table:landau}.
\item If $\alpha$ is not small enough, such as $\alpha=0.1$ or $\alpha=0.05$, it is advisable to choose $\gamma<1$ to avoid over-rejection while maintaining robustness against unknown dependence among the tests. Wilson's HMP with the Landau adjustment is another good option.
\end{enumerate}




\section{Finite sample sizes and powers}\label{sec:simulation}
This section compares various choices for the $p$-value combination methods in finite samples. 
The p-values $\bm{U}_n$ are generated from a normal-location model, similar to \cite{liu2020cauchy}'s simulation setting. That is, 
 $U_i=2-2\Phi(|Y_i|)$, where $Y_i\sim N(\bm{\mu},\Sigma)$ with all diagonal elements of the covariance matrix $\Sigma$ being 1 and all off-diagonal elements being $\rho$. The global null hypothesis is $\bm{\mu}=\bm{0}$. 
  We consider equal weights $w_i=1/n$. Weibull, Pareto, and log-Pareto calibrators are considered for the sum-combination statistics. Since all $F$ under consideration in this simulation have positive support, $S_{1,\mathcal{S},\ref{assump:w_equal}}(\bm{U}_n)=S_{2,\mathcal{S},\ref{assump:w_equal}}(\bm{U}_n)$, and only the results for $S_{1,\mathcal{S},\ref{assump:w_equal}}$ are reported. The max-combination statistic $S_{3,\mathcal{S},\ref{assump:w_equal}}(\bm{U}_n)$ is reported only once since its tail probability is invariant to the choice of $F$. For a Weibull calibrator the shape parameter $k$ is chosen as the upper limit in (\ref{eq:weibull:perfectcor2}), following Table \ref{table:optimal_weibull_logP}, and the scale parameter is 1. This sum-combination statistic is denoted as $W_v$. Weibull with a fixed $k$ was also compared but its over-rejection behavior was too severe and, therefore, is not reported here. For Pareto calibrators, we considered $\gamma=0.5,1,$ and 1.5. The sum-combining functions, $T_{1,F_{0.5}}$, $T_{1,F_{1}}$, and $T_{1,F_{1.5}}$, are presented, alongside with \cite{vovk2020combining}'s asymptotically precise combining functions, $M_{0.5}$, $M_1$, and $M_{1.5}$, as well as  Wilson's HMP with the Landau approximation $T_{1,W}$. For $M_{\gamma}$, $\widetilde{a}_{r,n}$ are chosen to satisfy Proposition 6 of \cite{vovk2020combining}, searching for the solutions on fine grids of their $y_K$. The chosen $\widetilde{a}_{r,n}$ values are 5.76, 7.45, and 10.11 for $n=25, 100,$ and 1000, respectively. $LP_5$ and $LP_v$ represent the ones with log-Pareto calibrators with $\gamma=5$ and with $\gamma$ chosen as in Table  \ref{table:optimal_weibull_logP}, respectively.
   For all combination methods except for $M_{\gamma}$,  the empirical rejection rates, finite sample analogs of $\bm{P}(S_{j,\mathcal{S},\ref{assump:w_equal}}(\bm{U}_n)>t_{\alpha,n})$, are computed using $t_{\alpha,n}$:
$$
t_{\alpha,n}=\left\{\begin{array}{ll}
n^{-1}\{-\log(\alpha n^{-1})\}^{\frac{1}{k}}, & {\rm if}~F~{\rm is~Weibull~with~shape~ } k {\rm and ~scale~}1,
\\\frac{n^{\frac{1}{\gamma}-1}}{\alpha^{\frac{1}{\gamma}}}, &{\rm if}~F~{\rm is~Pareto~with~index~} \gamma,
\\\frac{1}{n}\exp\{(\alpha n^{-1})^{-\frac{1}{\gamma}}\}&{\rm if}~F~{\rm is~log~Pareto~with~index~} \gamma.
\end{array}\right.$$
 The number of Monte Carlo replications is 1,000,000. We present the cases with $\rho=0,0.2,0.9$ to represent no, weak, and strong dependence. In particular, we chose $\rho=0.2$ to represent weak dependence because this seems the case where Wilson's HMP with the Landau adjustment shows the worst over-rejection behavior. This observation was also made by \cite{wilson2019reply} in their Fig 1 as a reply to \cite{goeman2019harmonic}'s argument for the over-rejection behavior of HMP in the presence of weak dependence with $\rho=0.2$. 

Tables \ref{table:size} and \ref{table:rawpower} report the empirical rejection rates under global null (size) and an alternative (power) of each combination method. For a better presentation of sizes, the size table, Table \ref{table:size}, presents the ratio between the raw sizes and the nominal ratio $\alpha$ rather than the raw sizes. This corresponds to $\bm{P}(S_{j,\mathcal{S},\ref{assump:w_equal}}(\bm{U}_n)>t_{\alpha,n})/\alpha$ for $W_v$, $T_{1,\gamma}$, $LP_5$, and $LP_{v}$, and $\bm{P}(S_{j,\mathcal{S},\ref{assump:w_equal}}(\bm{U}_n)>t_{\alpha,n})/\alpha_{wilson}$ for $T_{1,W}$, where $\alpha_{wilson}$ is the upper-tail $1/\alpha$th quantile of the Landau distribution $S(1,1,\pi/2,\log n + 0.874367)$. When this ratio is less than 1, it indicates that the corresponding test under-rejects, whereas if it is greater than 1, it over-rejects. The M-combining functions $M_{\gamma}$ are $\widetilde{a}_{r,n}M_{r,n}$ with $r=-1/\gamma$ and $\widetilde{a}$ chosen to make $M_{\gamma}$ asymptotically precise.
 For the power table, Table \ref{table:rawpower}, we considered a sparse alternative, where the first $s=\lfloor 0.05n\rfloor$ elements of $\bm{\mu}$ are $\mu_i=\sqrt{4\log n}/s^{0.1}$. The other elements $\mu_{s+1},\ldots,\mu_{n}$ of $\bm{\mu}$ are set as 0. Here, the symbol $\lfloor a\rfloor$ indicates the largest integer that is smaller than or equal to a real number $a$.

Our simulation results can be summarized as follows. First, when $\bm{U}_n$ are independent, the max-combination test $S_{1,\mathcal{S},\ref{assump:w_equal}}(\bm{U}_n)$, Wilson's HMP with the Landau adjustment, $T_{1,W}$, and the sum-combination test with $\gamma=0.5$, $T_{1,F_{0.5}}$, tend to perform the best. 
However, as $\rho$ increases, these tests tend to under-reject.
Second, the M-combining functions tend to under-rejection in all cases, which leads to some power losses. These power losses are sometimes quite drastic.
Third, our sum-combining function with Pareto $T_{1,F_{1}}$ shows some over-rejection when $\alpha=0.05$ or 0.1. The over-rejection is severer for larger $n$, smaller $\rho$, and larger $\alpha$. This behavior, particularly regarding $n$ and $\alpha$, is as expected in theory presented in Section \ref{sec:gammachoice}. 
Fourth, our sum-combining functions with Pareto $T_{1,\gamma}$ greatly depends on $\gamma$. When $\gamma=1.5$, the test $T_{1,F_{1.5}}$ over-rejects too seriously to provide any reliable test. When $\gamma=0.5$, on the contrary, provides quite reliable sizes and powers, similarly to Wilson's HMP with the Landau adjustment $T_{1,W}$.
Fifth, we can also see that the max-combining tests' sizes improve as $n$ increases when $\rho=0$, validating our proof in Section \ref{sec:ninft}. However, when $\rho=0.9$, the sizes get even worse (under-reject), and therefore, losing much power, as $n$ increases. 
Sixth, the log-Pareto with fixed $\gamma$, $LP_t$, does not seem cannot control the size with $\alpha=0.05$ and 0.1. However, when $\gamma$ is chosen according to our suggestion in (\ref{eq:example:gamma:logpareto}), the size seems to be best controlled, having more or less similar sizes as $T_{1,F_{1}}$.
Seventh, Weibull distribution does not control the size well even with $k$ chosen following (\ref{eq:weibull:perfectcor2}). We therefore do not recommend calibrators with relatively lighter tails, even if they are in the subexponential class.

Combining all results and observations in this paper, if the tests are known to be independent, the max-combining test, or the Bonferroni method, is the simplest with good sizes and powers. However, when dependence is suspected and the strength is unknown, we better take other measures. When $\alpha=0.01$ or smaller, unless the number of tests is extremely large, the sum-combination test with Pareto calibrator with index $\gamma=1$, Wilson's HMP with the Landau adjustment, or log-Pareto with varying $\gamma$ seems the best choices, especially when the dependence structure is either unknown or suspected to be quite strong. When relatively larger $\alpha$, 0.05 or 0.1, has to be chosen,  the decision depends on the number $n$ of the tests. If $n$ is relatively small with $n=25$ or $n=100$, Wilson's HMP with the Landau adjustment $T_{1,W}$ and Pareto with $\gamma=0.5$ $T_{1,F_{0.5}}$ provide the best-controlled sizes and reasonable powers when there is no or weak dependence. When $\rho=0.9$, both $T_{1,W}$ and $T_{1,F_{0.5}}$ under-rejects but $T_{1,W}$ may have better power.
 When relatively larger $\alpha$, 0.05 or 0.1, has to be chosen,  the decision depends on the number $n$ of the tests. If $n$ is relatively larger, $n=1000$, $T_{1,F_{0.5}}$ seems the best method for size control without sacrificing much power.


\begin{table}[ht]
\centering
\begin{tabular}{p{1mm}p{3mm}p{8mm}||r|rrrrrrr|rr|r}
\hline
\multirow{3}{*}{$\rho$}&\multirow{3}{*}{$n$}&\multirow{3}{*}{$\alpha$}&\multicolumn{10}{c|}{Sum}&\multirow{3}{*}{Max}\\\cline{4-13}
&&& W& \multicolumn{7}{c|}{Pareto}& \multicolumn{2}{c|}{log-Pareto}&\\\cline{4-13}
 & &  & $W_v$  & $T_{1,F_{0.5}}$ & $M_{0.5}$ & $T_{1,F_{1}}$ & $T_{1,W}$& $M_{1}$& $T_{1,F_{1.5}}$ & $M_{1.5}$ & $LP_v$ & $LP_5$ &  \\ 
   \hline\hline
\multirow{15}{*}{0}&\multirow{5}{*}{25}  & 0.1000 & 1.874 & 0.997 & 0.021 & 1.770 & 0.976 & 0.198 & 9.917 & 0.056 & 1.855 & 10.000 & 0.953 \\ 
&  & 0.0500 & 1.396 & 1.000 & 0.021 & 1.374 & 0.986 & 0.190 & 9.666 & 0.049 & 1.506 & 20.000 & 0.977 \\ 
&  & 0.0100 & 1.089 & 1.021 & 0.021 & 1.100 & 1.014 & 0.180 & 2.026 & 0.042 & 1.222 & 5.505 & 1.016 \\ 
&  & 0.0010 & 1.053 & 1.042 & 0.024 & 1.054 & 1.042 & 0.183 & 1.194 & 0.041 & 1.116 & 1.083 & 1.042 \\ 
&  & 0.0001 & 0.960 & 0.960 & 0.010 & 0.960 & 0.960 & 0.230 & 0.980 & 0.030 & 0.980 & 0.960 & 0.960 \\ 
   \cline{2-14}
&\multirow{5}{*}{100}   & 0.1000 & 3.226 & 1.001 & 0.005 & 2.248 & 0.995 & 0.150 & 10.000 & 0.027 & 2.051 & 10.000 & 0.955 \\ 
&   & 0.0500 & 1.851 & 1.001 & 0.005 & 1.534 & 0.997 & 0.142 & 20.000 & 0.024 & 1.586 & 20.000 & 0.978 \\ 
&   & 0.0100 & 1.103 & 0.995 & 0.004 & 1.090 & 0.991 & 0.133 & 4.167 & 0.021 & 1.214 & 2.709 & 0.990 \\ 
&   & 0.0010 & 0.986 & 0.979 & 0.007 & 0.989 & 0.977 & 0.144 & 1.233 & 0.017 & 1.061 & 0.985 & 0.979 \\ 
&   & 0.0001 & 1.090 & 1.080 & 0.030 & 1.090 & 1.080 & 0.120 & 1.150 & 0.030 & 1.140 & 1.080 & 1.080 \\ 
   \cline{2-14}
&\multirow{5}{*}{1000}& 0.1000 & 9.349 & 0.997 & 0.000 & 3.607 & 0.996 & 0.112 & 10.000 & 0.009 & 2.130 & 10.000 & 0.952 \\ 
  & & 0.0500 & 4.557 & 0.998 & 0.001 & 1.827 & 0.996 & 0.106 & 20.000 & 0.007 & 1.614 & 20.000 & 0.976 \\ 
  & & 0.0100 & 1.252 & 0.998 & 0.001 & 1.130 & 0.997 & 0.103 & 100.000 & 0.006 & 1.253 & 1.102 & 0.995 \\ 
  & & 0.0010 & 1.047 & 1.024 & 0.000 & 1.041 & 1.025 & 0.088 & 1.713 & 0.009 & 1.142 & 1.024 & 1.023 \\ 
  & & 0.0001 & 0.900 & 0.900 & 0.000 & 0.900 & 0.900 & 0.150 & 0.980 & 0.010 & 0.940 & 0.900 & 0.900 \\ 
   \hline
\multirow{15}{*}{0.2}&\multirow{5}{*}{25}& 0.1000 & 1.817 & 0.944 & 0.019 & 1.719 & 1.001 & 0.211 & 9.814 & 0.070 & 1.777 & 10.000 & 0.887 \\ 
 &  & 0.0500 & 1.497 & 0.968 & 0.018 & 1.446 & 1.052 & 0.198 & 8.514 & 0.055 & 1.532 & 20.000 & 0.926 \\ 
 &  & 0.0100 & 1.159 & 0.989 & 0.017 & 1.152 & 1.062 & 0.176 & 2.671 & 0.040 & 1.247 & 6.870 & 0.976 \\ 
 &  & 0.0010 & 0.968 & 0.920 & 0.015 & 0.964 & 0.956 & 0.145 & 1.232 & 0.038 & 1.019 & 0.971 & 0.915 \\ 
  & & 0.0001 & 0.870 & 0.870 & 0.020 & 0.870 & 0.870 & 0.150 & 0.900 & 0.020 & 0.900 & 0.870 & 0.850 \\ 
   \cline{2-14}
&\multirow{5}{*}{100}  & 0.1000 & 2.759 & 0.915 & 0.005 & 2.105 & 1.077 & 0.173 & 10.000 & 0.046 & 1.924 & 10.000 & 0.850 \\ 
  & & 0.0500 & 2.116 & 0.948 & 0.005 & 1.707 & 1.149 & 0.159 & 20.000 & 0.035 & 1.639 & 20.000 & 0.899 \\ 
  & & 0.0100 & 1.364 & 0.983 & 0.006 & 1.251 & 1.127 & 0.149 & 7.266 & 0.026 & 1.298 & 3.910 & 0.963 \\ 
  & & 0.0010 & 1.131 & 1.050 & 0.004 & 1.106 & 1.096 & 0.151 & 1.856 & 0.021 & 1.175 & 1.063 & 1.048 \\ 
& & 0.0001 & 1.190 & 1.170 & 0.010 & 1.200 & 1.190 & 0.140 & 1.310 & 0.030 & 1.220 & 1.170 & 1.170 \\ 
   \cline{2-14}
&\multirow{5}{*}{1000}  & 0.1000 & 5.554 & 0.848 & 0.000 & 2.728 & 1.189 & 0.151 & 10.000 & 0.030 & 1.897 & 10.000 & 0.774 \\ 
  & & 0.0500 & 4.144 & 0.898 & 0.000 & 2.223 & 1.342 & 0.134 & 20.000 & 0.017 & 1.652 & 20.000 & 0.837 \\ 
  & & 0.0100 & 2.285 & 0.973 & 0.000 & 1.529 & 1.347 & 0.108 & 95.950 & 0.009 & 1.355 & 1.123 & 0.942 \\ 
  & & 0.0010 & 1.246 & 0.991 & 0.000 & 1.105 & 1.082 & 0.110 & 7.437 & 0.005 & 1.141 & 0.985 & 0.978 \\ 
  & & 0.0001 & 1.100 & 1.090 & 0.000 & 1.100 & 1.100 & 0.070 & 1.630 & 0.000 & 1.140 & 1.080 & 1.080 \\ 
    \hline
\multirow{15}{*}{0.9}&\multirow{5}{*}{25} & 0.1000 & 1.115 & 0.298 & 0.008 & 1.086 & 0.712 & 0.189 & 5.269 & 0.186 & 1.071 & 10.000 & 0.192 \\ 
  & & 0.0500 & 1.122 & 0.312 & 0.008 & 1.090 & 0.836 & 0.190 & 5.182 & 0.183 & 1.072 & 20.000 & 0.207 \\ 
  & & 0.0100 & 1.124 & 0.350 & 0.008 & 1.095 & 1.018 & 0.198 & 4.968 & 0.182 & 1.072 & 10.538 & 0.254 \\ 
  & & 0.0010 & 1.167 & 0.401 & 0.006 & 1.135 & 1.123 & 0.187 & 4.701 & 0.163 & 1.102 & 0.784 & 0.299 \\ 
  & & 0.0001 & 1.120 & 0.420 & 0.010 & 1.120 & 1.120 & 0.220 & 4.430 & 0.110 & 1.110 & 0.410 & 0.310 \\ 
   \cline{2-14}
&\multirow{5}{*}{1000} & 0.1000 & 1.169 & 0.164 & 0.001 & 1.089 & 0.656 & 0.145 & 10.000 & 0.184 & 1.042 & 10.000 & 0.092 \\ 
  & & 0.0500 & 1.177 & 0.174 & 0.001 & 1.095 & 0.794 & 0.146 & 10.558 & 0.180 & 1.035 & 20.000 & 0.102 \\ 
  & & 0.0100 & 1.170 & 0.202 & 0.002 & 1.086 & 0.997 & 0.149 & 10.106 & 0.178 & 1.016 & 6.766 & 0.128 \\ 
  & & 0.0010 & 1.186 & 0.218 & 0.002 & 1.101 & 1.081 & 0.142 & 9.399 & 0.154 & 1.018 & 0.333 & 0.162 \\ 
  & & 0.0001 & 1.260 & 0.440 & 0.000 & 1.140 & 1.140 & 0.230 & 8.690 & 0.240 & 1.040 & 0.370 & 0.310 \\ 
   \cline{2-14}
&\multirow{5}{*}{1000}& 0.1000 & 1.251 & 0.061 & 0.000 & 1.088 & 0.577 & 0.109 & 10.000 & 0.186 & 1.000 & 10.000 & 0.027 \\ 
  & & 0.0500 & 1.264 & 0.066 & 0.000 & 1.095 & 0.735 & 0.109 & 20.000 & 0.182 & 0.990 & 20.000 & 0.031 \\ 
  & & 0.0100 & 1.279 & 0.075 & 0.000 & 1.098 & 0.989 & 0.110 & 32.991 & 0.178 & 0.964 & 0.370 & 0.038 \\ 
  & & 0.0010 & 1.312 & 0.091 & 0.000 & 1.112 & 1.102 & 0.113 & 30.827 & 0.158 & 0.928 & 0.070 & 0.056 \\ 
  & & 0.0001 & 1.300 & 0.110 & 0.000 & 1.140 & 1.140 & 0.130 & 29.030 & 0.170 & 0.950 & 0.060 & 0.050 \\ 
   \hline
\end{tabular}\caption{\footnotesize Ratios between raw sizes (empirical rejection rates under the global null) of combination statistics and their corresponding significance level $\alpha$ for various choices of the dependence strength $\rho$,  the number $n$ of tests, and the significance level $\alpha$. 1,000,000 Monte Carlo replications to obtain the tail probabilities. Since the tail probabilities for
the max-combination statistic $S_{3,\mathcal{S},\ref{assump:w_equal}}(\bm{U}_n)$ is invariant to the choice of calibrator $F$, reported only once in the last column. All the other columns are for the sum-combination statistics. The first column with "W" or $W_v$ is for a Weibull calibrator, with its shape parameter $k$ decided as the upper limit in (\ref{eq:weibull:perfectcor2}). The next seven columns are for Pareto calibrators: $T_{1,F_{0.5}}$, $T_{1,F_{1}}$, and $T_{1,F_{1.5}}$ are our sum-combining function $T_{1,F,w}(\bm{U}_n)$ in (\ref{eq:comb.fun.1}) with $\gamma=0.5,1,$ and $1.5$, respectively, and $M_{0.5}$, $M_{1}$, and $M_{1.5}$ are the asymptotically precise M-family merging functions  $\widetilde{a}_{r,n}M_{r,n}$ \citep{vovk2020combining} with $\gamma=-1/r=0.5,1,$ and $1.5$, respectively. The $T_{1,W}$ column is for Pareto 1 with Wilson's HMP with the Landau adjustment. The two columns with "log-Pareto" are using a log-Pareto calibrator with its index $\gamma$ chosen to satisfy (\ref{eq:example:gamma:logpareto}) ($LP_v$) or arbitrarily set as $\gamma=5$ ($LP_5$).
 }\label{table:size}
\end{table}

\begin{table}[ht]
\centering
\begin{tabular}{p{1mm}p{3mm}p{8mm}||r|rrrrrrr|rr|r}
\hline
\multirow{3}{*}{$\rho$}&\multirow{3}{*}{$n$}&\multirow{3}{*}{$\alpha$}&\multicolumn{10}{c|}{Sum}&\multirow{3}{*}{Max}\\\cline{4-13}
&&& W& \multicolumn{7}{c|}{Pareto}& \multicolumn{2}{c|}{log-Pareto}&\\\cline{4-13}
 & &  & $W_v$  & $T_{1,F_{0.5}}$ & $M_{0.5}$ & $T_{1,F_{1}}$ & $T_{1,W}$& $M_{1}$& $T_{1,F_{1.5}}$ & $M_{1.5}$ & $LP_v$ & $LP_5$ &  \\ 
   \hline\hline
\multirow{15}{*}{0}&\multirow{5}{*}{25}& 0.1000 & 0.859 & 0.789 & 0.361 & 0.855 & 0.787 & 0.600 & 1.000 & 0.458 & 0.861 & 1.000 & 0.783 \\ 
  & & 0.0500 & 0.747 & 0.708 & 0.302 & 0.746 & 0.707 & 0.516 & 0.953 & 0.377 & 0.757 & 1.000 & 0.705 \\ 
  & & 0.0100 & 0.531 & 0.523 & 0.189 & 0.532 & 0.523 & 0.350 & 0.600 & 0.235 & 0.543 & 0.717 & 0.523 \\ 
  & & 0.0010 & 0.303 & 0.302 & 0.089 & 0.303 & 0.302 & 0.182 & 0.313 & 0.112 & 0.307 & 0.305 & 0.302 \\ 
  & & 0.0001 & 0.152 & 0.152 & 0.038 & 0.153 & 0.152 & 0.084 & 0.154 & 0.049 & 0.154 & 0.152 & 0.152 \\ 
   \cline{2-14}
&\multirow{5}{*}{100}& 0.1000 & 1.000 & 0.996 & 0.640 & 1.000 & 0.998 & 0.969 & 1.000 & 0.899 & 1.000 & 1.000 & 0.994 \\ 
   & & 0.0500 & 0.998 & 0.989 & 0.552 & 0.997 & 0.994 & 0.934 & 1.000 & 0.820 & 0.996 & 1.000 & 0.985 \\ 
   & & 0.0100 & 0.964 & 0.935 & 0.362 & 0.956 & 0.952 & 0.790 & 0.995 & 0.591 & 0.953 & 0.986 & 0.927 \\ 
   & & 0.0010 & 0.768 & 0.726 & 0.170 & 0.755 & 0.754 & 0.494 & 0.825 & 0.298 & 0.748 & 0.728 & 0.717 \\ 
   & & 0.0001 & 0.464 & 0.439 & 0.069 & 0.457 & 0.457 & 0.246 & 0.498 & 0.126 & 0.452 & 0.437 & 0.435 \\ 
   \cline{2-14}
&\multirow{5}{*}{1000} & 0.1000 & 1.000 & 1.000 & 0.790 & 1.000 & 1.000 & 1.000 & 1.000 & 1.000 & 1.000 & 1.000 & 1.000 \\ 
   & & 0.0500 & 1.000 & 1.000 & 0.685 & 1.000 & 1.000 & 1.000 & 1.000 & 1.000 & 1.000 & 1.000 & 1.000 \\ 
   & & 0.0100 & 1.000 & 1.000 & 0.432 & 1.000 & 1.000 & 1.000 & 1.000 & 0.974 & 1.000 & 1.000 & 1.000 \\ 
   & & 0.0010 & 1.000 & 0.995 & 0.178 & 1.000 & 1.000 & 0.933 & 1.000 & 0.627 & 0.998 & 0.993 & 0.991 \\ 
   & & 0.0001 & 0.957 & 0.877 & 0.062 & 0.934 & 0.934 & 0.591 & 0.992 & 0.237 & 0.903 & 0.861 & 0.858 \\ 
   \hline
 \multirow{15}{*}{0.2} & \multirow{5}{*}{25}& 0.1000 & 0.859 & 0.791 & 0.362 & 0.856 & 0.789 & 0.600 & 1.000 & 0.458 & 0.862 & 1.000 & 0.784 \\ 
   & & 0.0500 & 0.747 & 0.709 & 0.302 & 0.748 & 0.708 & 0.516 & 0.949 & 0.377 & 0.759 & 1.000 & 0.705 \\ 
   & & 0.0100 & 0.531 & 0.524 & 0.190 & 0.533 & 0.524 & 0.351 & 0.600 & 0.236 & 0.544 & 0.718 & 0.523 \\ 
   & & 0.0010 & 0.303 & 0.302 & 0.089 & 0.303 & 0.302 & 0.182 & 0.313 & 0.112 & 0.308 & 0.305 & 0.302 \\ 
 & & 0.0001 & 0.153 & 0.153 & 0.038 & 0.153 & 0.153 & 0.085 & 0.154 & 0.049 & 0.154 & 0.153 & 0.153 \\ 
   \cline{2-14}
&\multirow{5}{*}{100}& 0.1000 & 1.000 & 0.985 & 0.567 & 0.999 & 0.993 & 0.924 & 1.000 & 0.822 & 0.998 & 1.000 & 0.980 \\ 
   & & 0.0500 & 0.994 & 0.965 & 0.489 & 0.989 & 0.979 & 0.870 & 1.000 & 0.734 & 0.986 & 1.000 & 0.958 \\ 
   & & 0.0100 & 0.916 & 0.872 & 0.324 & 0.902 & 0.896 & 0.706 & 0.986 & 0.525 & 0.897 & 0.962 & 0.864 \\ 
   & & 0.0010 & 0.685 & 0.646 & 0.156 & 0.673 & 0.672 & 0.440 & 0.738 & 0.274 & 0.666 & 0.648 & 0.637 \\ 
   & & 0.0001 & 0.414 & 0.391 & 0.065 & 0.408 & 0.408 & 0.225 & 0.445 & 0.121 & 0.403 & 0.389 & 0.387 \\ 
      \cline{2-14}
&\multirow{5}{*}{1000} & 0.1000 & 1.000 & 1.000 & 0.578 & 1.000 & 1.000 & 0.997 & 1.000 & 0.970 & 1.000 & 1.000 & 0.999 \\ 
   & & 0.0500 & 1.000 & 0.999 & 0.495 & 1.000 & 1.000 & 0.986 & 1.000 & 0.923 & 1.000 & 1.000 & 0.997 \\ 
   & & 0.0100 & 1.000 & 0.982 & 0.320 & 0.997 & 0.996 & 0.921 & 1.000 & 0.758 & 0.991 & 0.987 & 0.977 \\ 
   & & 0.0010 & 0.941 & 0.882 & 0.145 & 0.921 & 0.921 & 0.710 & 0.992 & 0.449 & 0.902 & 0.876 & 0.870 \\ 
   & & 0.0001 & 0.739 & 0.659 & 0.055 & 0.711 & 0.711 & 0.427 & 0.817 & 0.202 & 0.682 & 0.647 & 0.645 \\ 
   \hline
\multirow{15}{*}{0.9} & \multirow{5}{*}{25}&  0.1000 & 0.838 & 0.776 & 0.361 & 0.854 & 0.771 & 0.589 & 1.000 & 0.444 & 0.870 & 1.000 & 0.770 \\ 
   & & 0.0500 & 0.719 & 0.698 & 0.302 & 0.728 & 0.690 & 0.509 & 0.969 & 0.370 & 0.746 & 1.000 & 0.695 \\ 
   & & 0.0100 & 0.524 & 0.520 & 0.190 & 0.525 & 0.517 & 0.349 & 0.573 & 0.235 & 0.536 & 0.695 & 0.520 \\ 
   & & 0.0010 & 0.302 & 0.302 & 0.089 & 0.302 & 0.302 & 0.182 & 0.309 & 0.112 & 0.306 & 0.304 & 0.302 \\ 
   & & 0.0001 & 0.153 & 0.153 & 0.038 & 0.153 & 0.153 & 0.085 & 0.154 & 0.049 & 0.155 & 0.153 & 0.153 \\ 
    \cline{2-14}
&\multirow{5}{*}{100}& 0.1000 & 0.978 & 0.803 & 0.300 & 0.954 & 0.851 & 0.665 & 1.000 & 0.547 & 0.936 & 1.000 & 0.778 \\ 
   & & 0.0500 & 0.852 & 0.737 & 0.250 & 0.824 & 0.782 & 0.590 & 1.000 & 0.466 & 0.810 & 1.000 & 0.713 \\ 
   & & 0.0100 & 0.661 & 0.574 & 0.157 & 0.633 & 0.623 & 0.425 & 0.821 & 0.312 & 0.615 & 0.736 & 0.554 \\ 
   & & 0.0010 & 0.416 & 0.355 & 0.074 & 0.398 & 0.397 & 0.238 & 0.471 & 0.159 & 0.381 & 0.359 & 0.341 \\ 
 & & 0.0001 & 0.228 & 0.193 & 0.031 & 0.219 & 0.218 & 0.117 & 0.261 & 0.073 & 0.208 & 0.188 & 0.185 \\ 
    \cline{2-14}
&\multirow{5}{*}{1000} & 0.1000 & 1.000 & 0.703 & 0.120 & 0.942 & 0.818 & 0.617 & 1.000 & 0.520 & 0.862 & 1.000 & 0.654 \\ 
   & & 0.0500 & 0.884 & 0.635 & 0.096 & 0.811 & 0.759 & 0.542 & 1.000 & 0.441 & 0.743 & 1.000 & 0.588 \\ 
   & & 0.0100 & 0.690 & 0.477 & 0.055 & 0.618 & 0.606 & 0.383 & 1.000 & 0.290 & 0.546 & 0.511 & 0.438 \\ 
   & & 0.0010 & 0.432 & 0.283 & 0.023 & 0.384 & 0.382 & 0.208 & 0.585 & 0.144 & 0.328 & 0.269 & 0.258 \\ 
   & & 0.0001 & 0.234 & 0.149 & 0.009 & 0.209 & 0.208 & 0.100 & 0.332 & 0.064 & 0.175 & 0.137 & 0.135 \\ 
   \hline
\end{tabular}\caption{\footnotesize Raw powers (empirical rejection rates under an alternative) of combination statistics for various choices of the dependence strength $\rho$,  the number $n$ of tests, and the significance level $\alpha$.  1,000,000 Monte Carlo replications to obtain the tail probabilities. Since the tail probabilities for
the max-combination statistic $S_{3,\mathcal{S},\ref{assump:w_equal}}(\bm{U}_n)$ is invariant to the choice of calibrator $F$, reported only once in the last column. All the other columns are for the sum-combination statistics. The first column with "W" or $W_v$ is for a Weibull calibrator, with its shape parameter $k$ decided as the upper limit in (\ref{eq:weibull:perfectcor2}). The next seven columns are for Pareto calibrators: $T_{1,F_{0.5}}$, $T_{1,F_{1}}$, and $T_{1,F_{1.5}}$ are our sum-combining function $T_{1,F,w}(\bm{U}_n)$ in (\ref{eq:comb.fun.1}) with $\gamma=0.5,1,$ and $1.5$, respectively, and $M_{0.5}$, $M_{1}$, and $M_{1.5}$ are the asymptotically precise M-family merging functions  $\widetilde{a}_{r,n}M_{r,n}$ \citep{vovk2020combining} with $\gamma=-1/r=0.5,1,$ and $1.5$, respectively. The $T_{1,W}$ column is for Pareto 1 with Wilson's HMP with the Landau adjustment. The two columns with "log-Pareto" are using a log-Pareto calibrator with its index $\gamma$ chosen to satisfy (\ref{eq:example:gamma:logpareto}) ($LP_v$) or arbitrarily set as $\gamma=5$ ($LP_5$).}\label{table:rawpower}
\end{table}


\section*{Acknowledgement}
Rho acknowledges partial support from NSF-CPS grant \#1739422. The author extends gratitude to the participants at the 2022 IMS International Conference on Statistics and Data Science (ICSDS) for their valuable comments and discussions.

\bibliographystyle{chicago}
\bibliography{References.bib}

\begin{thebibliography}{}

\bibitem[\protect\citeauthoryear{Bae and Ko}{Bae and Ko}{2017}]{bae2017note}
Bae, T. and B.~Ko (2017).
\newblock A note on weighted infinite sums of dependent regularly varying tailed random variables.
\newblock {\em Journal of the Korean Statistical Society\/}~{\em 46\/}(3), 321--327.

\bibitem[\protect\citeauthoryear{Breiman}{Breiman}{1965}]{breiman1965some}
Breiman, L. (1965).
\newblock On some limit theorems similar to the arc-sin law.
\newblock {\em Theory of Probability \& Its Applications\/}~{\em 10\/}(2), 323--331.

\bibitem[\protect\citeauthoryear{Cang and Yang}{Cang and Yang}{2017}]{cang2017extremal}
Cang, Y. and Y.~Yang (2017).
\newblock On extremal behavior of aggregation of largest claims.
\newblock {\em Communications in Statistics-Theory and Methods\/}~{\em 46\/}(2), 917--926.

\bibitem[\protect\citeauthoryear{Chen, Liu, Tan, and Wang}{Chen et~al.}{2023}]{chen2023trade}
Chen, Y., P.~Liu, K.~S. Tan, and R.~Wang (2023).
\newblock Trade-off between validity and efficiency of merging p-values under arbitrary dependence.
\newblock {\em Statistica Sinica\/}~{\em 33}, 851--872.

\bibitem[\protect\citeauthoryear{Chen and Yuen}{Chen and Yuen}{2009}]{chen2009sums}
Chen, Y. and K.~C. Yuen (2009).
\newblock Sums of pairwise quasi-asymptotically independent random variables with consistent variation.
\newblock {\em Stochastic Models\/}~{\em 25\/}(1), 76--89.

\bibitem[\protect\citeauthoryear{Davis and Resnick}{Davis and Resnick}{1996}]{davis1996limit}
Davis, R.~A. and S.~I. Resnick (1996).
\newblock Limit theory for bilinear processes with heavy-tailed noise.
\newblock {\em The Annals of Applied Probability\/}~{\em 6\/}(4), 1191--1210.

\bibitem[\protect\citeauthoryear{Fang, Chang, Park, and Tseng}{Fang et~al.}{2024}]{fang2024heavy}
Fang, Y., C.~Chang, Y.~Park, and G.~C. Tseng (2024).
\newblock Heavy-tailed distribution for combining dependent p-values with asymptotic robustness.
\newblock {\em Statistica Sinica\/}, forthcoming.

\bibitem[\protect\citeauthoryear{Fisher}{Fisher}{1925}]{fisher1925}
Fisher, R.~A. (1925).
\newblock {\em Statistical methods for research workers}.
\newblock Oliver and Boyd.

\bibitem[\protect\citeauthoryear{Foss, Korshunov, and Zachary}{Foss et~al.}{2013}]{foss2013subexponential}
Foss, S., D.~Korshunov, and S.~Zachary (2013).
\newblock {\em An Introduction to Heavy-Tailed and Subexponential Distributions}.
\newblock Springer.

\bibitem[\protect\citeauthoryear{Gao, Liu, and Chai}{Gao et~al.}{2019}]{gao2019asymptotic}
Gao, Q., X.~Liu, and C.~Chai (2019).
\newblock Asymptotic tail probability of weighted infinite sum of conditionally dependent and consistently varying tailed random variables.
\newblock {\em Journal of Inequalities and Applications\/}~{\em 2019\/}(1), 1--15.

\bibitem[\protect\citeauthoryear{Geluk and Ng}{Geluk and Ng}{2006}]{geluk2006tail}
Geluk, J. and K.~W. Ng (2006).
\newblock Tail behavior of negatively associated heavy-tailed sums.
\newblock {\em Journal of applied probability\/}~{\em 43\/}(2), 587--593.

\bibitem[\protect\citeauthoryear{Geluk and Tang}{Geluk and Tang}{2009}]{geluk2009asymptotic}
Geluk, J. and Q.~Tang (2009).
\newblock Asymptotic tail probabilities of sums of dependent subexponential random variables.
\newblock {\em Journal of Theoretical Probability\/}~{\em 22\/}(4), 871--882.

\bibitem[\protect\citeauthoryear{Geng, Ji, and Wang}{Geng et~al.}{2019}]{geng2019tail}
Geng, B., R.~Ji, and S.~Wang (2019).
\newblock Tail probability of randomly weighted sums of dependent subexponential random variables with applications to risk theory.
\newblock {\em Journal of Mathematical Analysis and Applications\/}~{\em 480\/}(1), 123389.

\bibitem[\protect\citeauthoryear{Gnedenko}{Gnedenko}{1943}]{gnedenko1943distribution}
Gnedenko, B. (1943).
\newblock Sur la distribution limite du terme maximum d'une serie aleatoire.
\newblock {\em Annals of mathematics\/}, 423--453.

\bibitem[\protect\citeauthoryear{Goeman, Rosenblatt, and Nichols}{Goeman et~al.}{2019}]{goeman2019harmonic}
Goeman, J.~J., J.~D. Rosenblatt, and T.~E. Nichols (2019).
\newblock The harmonic mean p-value: Strong versus weak control, and the assumption of independence.
\newblock {\em Proceedings of the National Academy of Sciences\/}~{\em 116\/}(47), 23382--23383.

\bibitem[\protect\citeauthoryear{Jiang, Gao, and Wang}{Jiang et~al.}{2014}]{jiang2014max}
Jiang, T., Q.~Gao, and Y.~Wang (2014).
\newblock Max-sum equivalence of conditionally dependent random variables.
\newblock {\em Statistics \& Probability Letters\/}~{\em 84}, 60--66.

\bibitem[\protect\citeauthoryear{Ko and Tang}{Ko and Tang}{2008}]{ko2008sums}
Ko, B. and Q.~Tang (2008).
\newblock Sums of dependent nonnegative random variables with subexponential tails.
\newblock {\em Journal of Applied Probability\/}~{\em 45\/}(1), 85--94.

\bibitem[\protect\citeauthoryear{Li}{Li}{2013}]{li2013pairwise}
Li, J. (2013).
\newblock On pairwise quasi-asymptotically independent random variables and their applications.
\newblock {\em Statistics \& Probability Letters\/}~{\em 83\/}(9), 2081--2087.

\bibitem[\protect\citeauthoryear{Ling and Rho}{Ling and Rho}{2022}]{ling2022stable}
Ling, X. and Y.~Rho (2022).
\newblock Stable combination tests.
\newblock {\em Statistica Sinica\/}~{\em 32}, 641--644.

\bibitem[\protect\citeauthoryear{Liu}{Liu}{2009}]{liu2009precise}
Liu, L. (2009).
\newblock Precise large deviations for dependent random variables with heavy tails.
\newblock {\em Statistics \& Probability Letters\/}~{\em 79\/}(9), 1290--1298.

\bibitem[\protect\citeauthoryear{Liu and Xie}{Liu and Xie}{2020}]{liu2020cauchy}
Liu, Y. and J.~Xie (2020).
\newblock Cauchy combination test: a powerful test with analytic p-value calculation under arbitrary dependency structures.
\newblock {\em Journal of the American Statistical Association\/}~{\em 115\/}(529), 393--402.

\bibitem[\protect\citeauthoryear{Nair, Wierman, and Zwart}{Nair et~al.}{2022}]{nair2022fundamentals}
Nair, J., A.~Wierman, and B.~Zwart (2022).
\newblock {\em The fundamentals of heavy tails: Properties, emergence, and estimation}, Volume~53.
\newblock Cambridge University Press.

\bibitem[\protect\citeauthoryear{Nolan}{Nolan}{2020}]{nolan2020univariate}
Nolan, J.~P. (2020).
\newblock Univariate stable distributions.
\newblock {\em Springer Series in Operations Research and Financial Engineering, DOI\/}~{\em 10}, 978--3.

\bibitem[\protect\citeauthoryear{Qian, Geng, and Wang}{Qian et~al.}{2022}]{qian2022tail}
Qian, H., B.~Geng, and S.~Wang (2022).
\newblock Tail asymptotics of randomly weighted sums of dependent strong subexponential random variables.
\newblock {\em Lithuanian Mathematical Journal\/}~{\em 62\/}(1), 113--122.

\bibitem[\protect\citeauthoryear{Resnick}{Resnick}{1987}]{resnick2008extreme}
Resnick, S.~I. (1987).
\newblock {\em Extreme values, regular variation, and point processes}.
\newblock Springer, New York.

\bibitem[\protect\citeauthoryear{Sibuya et~al.}{Sibuya et~al.}{1960}]{sibuya1960bivariate}
Sibuya, M. et~al. (1960).
\newblock Bivariate extreme statistics.
\newblock {\em Annals of the Institute of Statistical Mathematics\/}~{\em 11\/}(2), 195--210.

\bibitem[\protect\citeauthoryear{Stouffer, Suchman, DeVinney, Star, and Williams~Jr}{Stouffer et~al.}{1949}]{stouffer1949american}
Stouffer, S.~A., E.~A. Suchman, L.~C. DeVinney, S.~A. Star, and R.~M. Williams~Jr (1949).
\newblock The american soldier: Adjustment during army life.(studies in social psychology in world war ii), vol. 1.

\bibitem[\protect\citeauthoryear{Tang and Tsitsiashvili}{Tang and Tsitsiashvili}{2003}]{tang2003randomly}
Tang, Q. and G.~Tsitsiashvili (2003).
\newblock Randomly weighted sums of subexponential random variables with application to ruin theory.
\newblock {\em Extremes\/}~{\em 6\/}(3), 171--188.

\bibitem[\protect\citeauthoryear{Uchaikin and Zolotarev}{Uchaikin and Zolotarev}{2011}]{uchaikin2011chance}
Uchaikin, V.~V. and V.~M. Zolotarev (2011).
\newblock {\em Chance and stability: stable distributions and their applications}.
\newblock Walter de Gruyter.

\bibitem[\protect\citeauthoryear{Vovk, Wang, and Wang}{Vovk et~al.}{2022}]{vovk2022admissible}
Vovk, V., B.~Wang, and R.~Wang (2022).
\newblock Admissible ways of merging p-values under arbitrary dependence.
\newblock {\em The Annals of Statistics\/}~{\em 50\/}(1), 351--375.

\bibitem[\protect\citeauthoryear{Vovk and Wang}{Vovk and Wang}{2020}]{vovk2020combining}
Vovk, V. and R.~Wang (2020, 06).
\newblock {Combining p-values via averaging}.
\newblock {\em Biometrika\/}~{\em 107\/}(4), 791--808.

\bibitem[\protect\citeauthoryear{Vovk and Wang}{Vovk and Wang}{2021}]{vovk2021values}
Vovk, V. and R.~Wang (2021).
\newblock E-values: Calibration, combination and applications.
\newblock {\em The Annals of Statistics\/}~{\em 49\/}(3), 1736--1754.

\bibitem[\protect\citeauthoryear{Wang and Tang}{Wang and Tang}{2006}]{wang2006tail}
Wang, D. and Q.~Tang (2006).
\newblock Tail probabilities of randomly weighted sums of random variables with dominated variation.
\newblock {\em Stochastic models\/}~{\em 22\/}(2), 253--272.

\bibitem[\protect\citeauthoryear{Wilson}{Wilson}{2019a}]{wilson2019harmonic}
Wilson, D.~J. (2019a).
\newblock The harmonic mean p-value for combining dependent tests.
\newblock {\em Proceedings of the National Academy of Sciences\/}~{\em 116\/}(4), 1195--1200.

\bibitem[\protect\citeauthoryear{Wilson}{Wilson}{2019b}]{wilson2019reply}
Wilson, D.~J. (2019b).
\newblock Reply to goeman et al.: Trade-offs in model averaging using multilevel tests.
\newblock {\em Proceedings of the National Academy of Sciences\/}~{\em 116\/}(47), 23384--23385.

\bibitem[\protect\citeauthoryear{Wilson}{Wilson}{2020}]{wilson2020generalized}
Wilson, D.~J. (2020).
\newblock Generalized mean p-values for combining dependent tests: comparison of generalized central limit theorem and robust risk analysis.
\newblock {\em Wellcome Open Research\/}~{\em 5}.

\bibitem[\protect\citeauthoryear{Yi, Chen, and Su}{Yi et~al.}{2011}]{yi2011approximation}
Yi, L., Y.~Chen, and C.~Su (2011).
\newblock Approximation of the tail probability of randomly weighted sums of dependent random variables with dominated variation.
\newblock {\em Journal of Mathematical Analysis and Applications\/}~{\em 376\/}(1), 365--372.

\end{thebibliography}

\appendix

\section*{APPENDICES}

\begin{proof}[Proof of Theorem \ref{thm:bivariate_normal}]
Without loss of generality, assume $F$ has positive support. 
Let $y_1=\overline{\Phi}^{-1}\left(\frac{1}{2}\overline{F}(x_1)\right)$ and $y_2=\overline{\Phi}^{-1}\left(\frac{1}{2}\overline{F}(x_2)\right)$, which are functions in $x_1$ and $x_2$, respectively. Since both $\overline{F}$ and $\overline{\Phi}^{-1}$ are decreasing functions,
$\bm{P}(X_1>x_1)=\bm{P}\left(|Y_1|>\overline{\Phi}^{-1}\left(\frac{1}{2}\overline{F}(x_1)\right)\right)=\bm{P}(|Y_1|>y_1)=2\overline{\Phi}(y_1)$. 

The joint density of $|Y_1|$ and $|Y_2|$ is, for $y_1>0$ and $y_2>0$,
$$f_{|Y_1|,|Y_2|}(y_1,y_2)=\frac{1}{2\pi\sqrt{1-\rho^2}}\left\{2\exp\left(-\frac{y_1^2+y_2^2-2\rho y_1 y_2}{2(1-\rho^2)}\right)+2\exp\left(-\frac{y_1^2+y_2^2+2\rho y_1 y_2}{2(1-\rho^2)}\right)\right\}$$
and the marginal density of $|Y_2|$ is
$$f_{|Y_2|}(y_2)=\frac{2}{\sqrt{2\pi}}\exp(-y_2^2/2),~~~y_2>0.$$
The conditional density of $|Y_1|$ given $|Y_2|$ is
$$f_{|Y_1| \big| |Y_2|}(y_1|y_2)=\frac{1}{\sqrt{2\pi(1-\rho^2)}}\left\{\exp\left(-\frac{(y_1-\rho y_2)^2}{2(1-\rho^2)}\right)+\exp\left(-\frac{(y_1+\rho y_2)^2}{2(1-\rho^2)}\right)\right\},$$
and therefore,
\begin{equation}\label{eq:prof_multnorm1}
\bm{P}\left(|Y_1|>y_1\big| |Y_2|=y_2\right)=\overline{\Phi}\left(\frac{y_1-\rho y_2}{\sqrt{1-\rho^2}}\right)+\overline{\Phi}\left(\frac{y_1+\rho y_2}{\sqrt{1-\rho^2}}\right).\end{equation}

Suppose that Assumption \ref{assmp:geluktang} or \ref{assmp:kotang} is satisfied, or equivalently, there exists $t_0>0$ such that
\begin{equation}\label{eq:proof1-ratio}
\frac{\bm{P}\left(|Y_1|>y_1\big| |Y_2|=y_2\right)}{\bm{P}\left(|Y_1|>y_1 \right)} =\frac{\overline{\Phi}\left(\frac{y_1-\rho y_2}{\sqrt{1-\rho^2}}\right)+\overline{\Phi}\left(\frac{y_1+\rho y_2}{\sqrt{1-\rho^2}}\right)}{2\overline{\Phi}(y_1)}
\end{equation}
is uniformly bounded for any $x_1>t_0$ and $x_2>t_0$, or for $x_1\in[0,t-t_0]$ and $x_2\in[t_0,t]$ whenever $t>t_0$, respectively. Since $y_1$ and $y_2$ are increasing functions in $x_1$ and $x_2$ without any upper limit, one can always find $y_2$ for any given $y_1$ such that $|\rho|y_2>2y_1$ by setting $x_2$ much larger than $x_1$. In addition, for any arbitrarily small $\epsilon>0$, we can always find $y_1$ such that $\overline{\Phi}(y_1)<\epsilon$. Therefore, $$\overline{\Phi}\left(\frac{y_1-|\rho| y_2}{\sqrt{1-\rho^2}}\right)>\overline{\Phi}\left(\frac{-y_1}{\sqrt{1-\rho^2}}\right)={\Phi}\left(\frac{y_1}{\sqrt{1-\rho^2}}\right)>{\Phi}(y_1)>1-\epsilon,$$
and hence, a lower bound of the right-hand side of (\ref{eq:proof1-ratio}) is $(1-\epsilon)/2\epsilon$, which can be arbitrarily large. This means that such $t_0$ cannot exist, or equivalently, Assumption \ref{assmp:geluktang} nor \ref{assmp:kotang} can be satisfied.
\end{proof}

\begin{proof}[Proof of Theorem \ref{thm:inS_rho1}]
The proof for $S_{1,F,w}(\bm{U}_n)$ is straightforward  considering that $\sum_{i=1}^nw_iX_i=\sum_{i=1}^{m+1}{w}^{*}_iX_i$ with ${w}^{*}_i=w_i$ for $i=1,\ldots,m$ and ${w}^{*}_{m+1}=\sum_{i=m+1}^nw_i$.
For $S_{3,F,w}(\bm{U}_n)$,
$\vee_{i=1}^nw_iX_i=\vee_{i=1}^{m+1}{w}^{**}_iX_i$, where ${w}^{**}_i=w_i$ for $i=1,\ldots,m$ and ${w}^{**}_{m+1}=\vee_{i=m+1}^nw_i$.  Since $\bm{E}w_i>0$ for all $i=1,\ldots,n$, $\sum_{i=1}^{m+1}\bm{E}w^{**}_i<\sum_{i=1}^n\bm{E}w_i=1$ whenever $m\leq n-2$. Using the rescaled weights $w^{**}_i\left(\sum_{j=1}^{m+1}\bm{E}w^{**}_j\right)^{-1}$ for $i=1,\ldots,m+1$, we obtain the desired result based on Theorem \ref{thm:inS} (\ref{thm:inS-2}).
\end{proof}

The following Lemmas are useful throughout the proof in this paper.
\begin{lemma}\label{lemma1}
 For a nonrandom number $w\in(0,1)$,  $w^{r_1}> w^{r_2}$ for any $0<r_1<r_2<\infty$.
 \end{lemma}
 \begin{lemma}\label{lemma2}
      For a nonnegative random variable $w$ with $\bm{E}(0<w<1)=1$, $ \bm{E}w^{r_2}< \bm{E}w^{r_1}<1$ for any $0<r_1<r_2<\infty$.
\end{lemma}
\begin{proof}[Proofs of Lemmas \ref{lemma1} and \ref{lemma2}]
Lemma \ref{lemma1} is trivial. Lemma \ref{lemma2} is a straightforward result of Lemma \ref{lemma1} by observing $\int_0^1\omega^{r_2}f_w(\omega)d\omega<\int_0^1\omega^{r_1}f_w(\omega)d\omega<\int_0^1f_w(\omega)d\omega=1$ when $\omega^{r_2}<\omega^{r_1}<1$ for any $\omega\in(0,1)$. Here,  $f_w(\omega)$ is the density function of $w$.
\end{proof}

\begin{lemma}\label{lemma:wigamma}
    Let $0\leq w_i\leq 1$ for all $i=1,\ldots,n$ and $\sum_{i=1}^nw_i=1$. For any positive integer $n$,
    \begin{enumerate}[{\normalfont (i)}]
        \item $1\leq\sum_{i=1}^n w_i^r\leq n^{1-r}$ for any $0<r<1$;
         \item $0<n^{1-r}\leq\sum_{i=1}^n w_i^r\leq1$ when $r>1$. 
    \end{enumerate}
    The upper or lower bound $n^{1-r}$ of $\sum_{i=1}^nw_i^r$ is achieved when $w_i=n^{-1}$. In addition, if the trivial case with $w_i=1$ for some $i$ and $w_j=0$ for all $j\neq i$ is excluded, the inequalities involving 1 above become strict: $\sum_{i=1}^nw_i^r>1$ when $0<r<1$ and  $\sum_{i=1}^nw_i^r<1$ when $r>1$. 
\end{lemma}

\begin{proof}[Proof of Lemma \ref{lemma:wigamma}]
Consider $0<r<1$. Define $p_r=1/r$ and $q_r=1/(1-r)$ so that $1/p_r+1/q_r=1$ with $p_r,q_r>1$. Using H\"older's inequality, 
$$\sum_{i=1}^nw_i^r\leq\left(\sum_{i=1}^n|w_i^r|^{p_r}\right)^{\frac{1}{p_r}}\left(\sum_{i=1}^n1^{q_r}\right)^{\frac{1}{q_r}}=\left(\sum_{i=1}^nw_i\right)^{r}n^{1-r}=n^{1-r}.$$
The equality holds if $w_1=\cdots=w_n$.

When $r>1$, H\"older's inequality implies
$$\sum_{i=1}^nw_i\leq\left(\sum_{i=1}^nw_i^r\right)^{\frac{1}{r}}\left(\sum_{i=1}^n1^{\frac{r}{r-1}}\right)^{\frac{r-1}{r}},$$
which can be simplified as
$$1\leq \left(\sum_{i=1}^nw_i^r\right)^{\frac{1}{r}}n^{\frac{r-1}{r}}$$
by plugging in $\sum_{i=1}^nw_i=1$. Therefore, $n^{r-1}\sum_{i=1}^nw_i^r\geq1$, which leads to $\sum_{i=1}^n w_i^r\geq n^{1-r}$. Again, the equality holds when $w_1=\cdots=w_n$. 

The other inequalities $\sum_{i=1}^nw_i^r>1$ when $r\in(0,1)$ and $\sum_{i=1}^nw_i^r<1$ when $r>1$ are straightforward result from Leamma \ref{lemma1} and $\sum_{i=1}^nw_i=1$, whenever all $w_i$ avoid 0 and 1. 
In the trivial case with $w_i$ being either 0 or 1, $\sum_{i=1}^nw_i^r=1$ regardless the value of $r$. 
\end{proof}

\begin{lemma}\label{lemma:wi_jensen}
For any nonnegative random variable $w$, $(\bm{E}w^{r_1})^{\frac{1}{r_1}}\leq(\bm{E}w^{r_2})^{\frac{1}{r_2}}$ for any $0<r_1<r_2<\infty$. The equality holds only when $w$ is nonrandom. 
\end{lemma}
\begin{proof}
This is a straightforward result of applying Jensen's inequality with a convex function $\varphi(x)=x^{\frac{r_2}{r_1}}$ to $w^{r_1}$. That is, $\varphi\left(\bm{E}w^{r_1}\right)=\left(\bm{E}w^{r_1}\right)^{\frac{r_2}{r_1}}\leq \bm{E}\left(\varphi(w)\right)=\bm{E}\left(w^{r_1\frac{r_2}{r_1}}\right)=\bm{E}w^{r_2}$. Since $\varphi(x)$ is not a linear function, the equality holds only when $w$ is nonrandom. 
\end{proof}

\begin{lemma}\label{lemma:wigamma_random}
   Consider a nonnegative random sequence $\{w_i\}_{i=1}^n$ that satisfies ASsumption \ref{assump:w_random}. That is, $\bm{P}(0<w_i<1)=1$ and $\sum_{i=1}^n\bm{E}w_i=1.$  For any positive integer $n$,
    \begin{enumerate}[{\normalfont (i)}]
        \item $1\leq\sum_{i=1}^n \bm{E}w_i^r \leq n^{1-r}$ for any   $0<r<1$; 
         \item $0<n^{1-r}\leq\sum_{i=1}^n \bm{E}w_i^r\leq1$ for any $r>1$; 
         \item\label{lemma:w_random_item3} $1\leq \sum_{i=1}^n\left(\bm{E}w_i^r\right)^{1/r}\leq n^{1-\frac{1}{r}}$ for any $r>1$.
    \end{enumerate}
    The upper or lower bound $n^{1-\gamma}$ of $\sum_{i=1}^nw_i^\gamma$  is achieved only when all $w_i$ are nonrandom and  $w_i=n^{-1}$ for all $i$. In addition, if the trivial case with $\bm{E}w_i=1$ for some $i$ and $\bm{E}w_j=0$ for all $j\neq i$ is excluded, the inequalities involving 1 also become strict: $\sum_{i=1}^n\bm{E}w_i^\gamma>1$ when $0<\gamma<1$ and  $\sum_{i=1}^n\bm{E}w_i^\gamma<1$ when $\gamma>1$.    
\end{lemma}

\begin{proof}[Proof of Lemma \ref{lemma:wigamma}] 
 Consider $0<r<1$. Define $p_r=1/r$ and $q_r=1/(1-r)$ so that $1/p_r+1/q_r=1$ with $p_r,q_r>1$. Using H\"older's inequality and Lemma \ref{lemma:wi_jensen}, 
$$\sum_{i=1}^n\bm{E}w_i^r\leq\left(\sum_{i=1}^n\left(\bm{E}w_i^r\right)^{p_r}\right)^{\frac{1}{p_r}}\left(\sum_{i=1}^n1^{q_r}\right)^{\frac{1}{q_r}}=\left(\sum_{i=1}^n\left(\bm{E}w_i^r\right)^{\frac{1}{r}}\right)^{r}n^{1-r}\leq \left(\sum_{i=1}^n\bm{E}w_i\right)^rn^{1-r}=n^{1-r}.$$
Due to the equality condition given in Lemma \ref{lemma:wi_jensen}, this inequality is strict unless all $w_i=n^{-1}$ and nonrandom.

When $r>1$, Lemma \ref{lemma:wi_jensen} and H\"older's inequality imply
\begin{equation}\label{eq:lemma:randomw_holder}1=\sum_{i=1}^n\bm{E}w_i\leq \sum_{i=1}^n\left(\bm{E}w_i^r\right)^{\frac{1}{r}}\leq\left(\sum_{i=1}^n\left\{\left(\bm{E}w_i^r\right)^{\frac{1}{r}}\right\}^{r}\right)^{\frac{1}{r}}\left(\sum_{i=1}^n1^{r/(r-1)}\right)^{\frac{r-1}{r}},\end{equation}
which can be simplified as
$$1\leq \left(\sum_{i=1}^n\bm{E}w_i^r\right)^{\frac{1}{r}}n^{\frac{r-1}{r}}.$$
Therefore, $\sum_{i=1}^n \bm{E}w_i^r\geq n^{1-r}$.
Again, this inequality is strict unless all $w_i=n^{-1}$ and nonrandom.

The inequalities $\sum_{i=1}^n\bm{E}w_i^r>1$ when $r\in(0,1)$ and $\sum_{i=1}^n\bm{E}w_i^r<1$ when $r>1$ are direct consequences of Lemma \ref{lemma2} as long as all $\bm{E}w_i$ avoid 0 and 1.
In the trivial case with $\bm{E}w_i$ being either 0 or 1, $\sum_{i=1}^n\bm{E}w_i^r=1$ regardless the value of $r$.

Now we prove (\ref{lemma:w_random_item3}). The lower bound is directly from Lemma \ref{lemma:wi_jensen} or (\ref{eq:lemma:randomw_holder}). The upper bound is  from  H\"older's inequality in (\ref{eq:lemma:randomw_holder}) and the fact that  $\sum_{i=1}^n\bm{E}w_i^r<1$ when $r>1$:
$$\sum_{i=1}^n\left(\bm{E}w_i^r\right)^{\frac{1}{r}}\leq \left(\sum_{i=1}^n\bm{E}w_i^r\right)^{\frac{1}{r}}n^{\frac{r-1}{r}}\leq n^{1-\frac{1}{r}}.$$
\end{proof}

Now we prove some main theorems of this paper.

\begin{proof}[Proof of Theorem \ref{thm:ninfty_gao}]
 Let $\alpha_n=C\epsilon_n n^{-\tau}$ for some $\epsilon_n\to0$ as $n\to\infty$ and positive real numbers $C$ and $\tau$ that do not depend on $n$.
The tail probability of the sum-combination statistic with the lower threshold $t_{\alpha,n}$ such that $\overline{F}(t_{\alpha,n})=\alpha_n(\sum_{i=1}^nw_i^\gamma)^{-1}=C\epsilon_n n^{-\tau}(\sum_{i=1}^nw_i^\gamma)^{-1}$
can be written as
\begin{equation}\label{eq1_proof_thm5-4}
\begin{array}{ll}\bm{P}(S_{1,F,w}(\bm{U}_n)>t_{\alpha,n})
&=\bm{P}\left(n^\tau(\sum_{i=1}^nw_i^\gamma)\overline{F}(S_{1,F,w}(\bm{U}_n))<C\epsilon_n\right)
\\[2mm]&\sim\bm{P}\left(\overline{F}\left(\sum_{i=1}^n \frac{w_i}{n^{\frac{\tau}{\gamma}}(\sum_{j=1}^n w_j^\gamma)^{\frac{1}{\gamma}}}X_i\right)<C\epsilon_n\right)
\\[2mm]&=\bm{P}\left(\sum_{i=1}^n\widetilde{w}_i X_i>\overline{F}^{-1}(C\epsilon_n)\right),\end{array}\end{equation}
where
\begin{equation}\label{wtilde}\widetilde{w}_i=\frac{w_i}{n^{\frac{\tau}{\gamma}}(\sum_{j=1}^n w_j^\gamma)^{\frac{1}{\gamma}}}.\end{equation}

Let $0<q_1<\min\{1,\gamma\}$.
By taking the exponent $q_1>0$ on both sides of (\ref{wtilde}) and  summing over $i=1,\ldots,n$,
$$\sum_{i=1}^n\widetilde{w}_i^{q_1}=\frac{\sum_{i=1}^n w_i^{q_1}}{n^{\frac{\tau q_1}{\gamma}}(\sum_{j=1}^n w_j^\gamma)^{\frac{q_1}{\gamma}}}.$$
Since $0<q_1<1$ for any $\gamma>0$, $\sum_{i=1}^n w_i^{q_1}\leq n^{1-q_1}$ by Lemma \ref{lemma:wigamma}, with the equality holding when $w_i=n^{-1}$ for all $i$.

When $0<\gamma<1$,  the lower bound of $\sum_{j=1}^nw_j^\gamma$ is 1 from Lemma \ref{lemma:wigamma}, which leads to
$$\sum_{i=1}^n\widetilde{w}_i^{q_1}\leq \frac{n^{1-q_1}}{n^{\frac{\tau q_1}{\gamma}}}
=n^{1-q_1-\frac{\tau q_1}{\gamma}}\leq 1$$
uniformly for all $n$
if $1-q_1-\frac{\tau q_1}{\gamma}\leq 0,$
or equivalently,
$\tau\geq \frac{\gamma}{q_1}-\gamma.$

 When $\gamma\geq 1$, the lower bound of $\sum_{j=1}^nw_j^\gamma$ is $n^{1-\gamma}$ from Lemma \ref{lemma:wigamma}, and therefore,
$$\sum_{i=1}^n\widetilde{w}_i^{q_1}\leq \frac{n^{1-q_1}}{n^{\frac{\tau q_1}{\gamma}}n^{\frac{(1-\gamma)q_1}{\gamma}}}
=n^{1-q_1-\frac{\tau q_1}{\gamma}-\frac{q_1-\gamma q_1}{\gamma}}=n^{1-\frac{\tau q_1}{\gamma}-\frac{ q_1}{\gamma}}\leq 1$$
uniformly for all $n$
if $1-\frac{\tau q_1}{\gamma}-\frac{ q_1}{\gamma}\leq 0,$
or equivalently,
$\tau\geq \frac{\gamma}{q_1}-1.$ 

Hence, for any $\gamma>0$, $$\lim_{n\to\infty}\sum_{i=1}^n \widetilde{w}_i^{q_1}\leq 1<\infty$$
 if $\tau\geq\frac{\gamma}{q_1}-\min\{1,\gamma\}.$ Since the lower bound $\frac{\gamma}{q_1}-\min\{1,\gamma\}$ is increasing in $q_1$, this condition is equivalent to $\tau>|1-\gamma|$ by choosing $q_1$ as close to $\gamma$ as possible when $\gamma\in(0,1)$ and close to $1$ when $\gamma\geq1$.
In this case,  $\alpha_n=C\epsilon_nn^{-\tau}<C\epsilon_nn^{-|1-\gamma|}$, which is equivalent to Assumption \ref{assump:alpha_nonrandw}.

With  the assumption $\min_{i=1,\ldots,n}w_i\geq c_0n^{-1}$ for some constant $c_0>0$ as in (\ref{assum:minweight}), for any $\tau>0$, $$\sum_{i=1}^n\widetilde{w}_i^{q_1}\leq\frac{n^{1-q_1}}{n^{\frac{\tau q_1}{\gamma}}n^{\frac{(1-\gamma)q_1}{\gamma}}c_0^{q_1}}=c_0^{-q_1}n^{1-\frac{\tau q_1}{\gamma}-\frac{ q_1}{\gamma}}\leq c_0^{-q_1}<\infty$$
uniformly for all $n$
if $1-\frac{\tau q_1}{\gamma}-\frac{ q_1}{\gamma}\leq 0,$
or equivalently,
$\tau\geq \frac{\gamma}{q_1}-1.$ Since $\frac{\gamma}{q_1}-1$ is increasing in $q_1$, this condition is equivalent to $\tau>(\gamma-1)\mathbb{I}_{\{\gamma>1\}}$ by choosing $q_1$ as close to its upper limit $\min\{1,\gamma\}$ as possible. Therefore, $\alpha_n=C\epsilon_nn^{-\tau}<C\epsilon_nn^{-(\gamma-1)\mathbb{I}_{\{\gamma>1\}}}$, which is equivalent to Assumption \ref{assump:alpha_equalw}.

Proposition \ref{proposition:infinite_n_gao} implies that
$$\bm{P}\left(\sum_{i=1}^n\widetilde{w}_iX_i>\overline{F}^{-1}(C\epsilon_n)\right)\sim C\epsilon_n\left(\sum_{i=1}^\infty\widetilde{w}_i^\gamma\right) 
$$
holds as $\overline{F}^{-1}(C\epsilon_n)\to\infty$, or equivalently, as $n\to\infty$.
In addition, from the definition of $\widetilde{w}_i$ in equation (\ref{wtilde}), 
$\sum_{i=1}^n\widetilde{w}_i^{\gamma}=n^{-\tau}=(C\epsilon_n)^{-1}\alpha_n$, which leads to
$$\frac{C\epsilon_n\sum_{i=1}^n \widetilde{w}_i^\gamma}{\alpha_n}=1$$
for any positive integer $n$. By taking the limit as $n\to\infty$ on both sides, we have $C\epsilon_n\sum_{i=1}^\infty \widetilde{w}_i^\gamma\sim \alpha_n$. Hence
$\bm{P}(S_{1,F,w}(\bm{U}_n)>t_{\alpha,n})\sim\alpha_n.$
With $a_{n,\gamma}=\left(\sum_{i=1}^nw_i^\gamma\right)^{-\frac{1}{\gamma}}$ as defined in (\ref{a_n}),
from the first line of (\ref{eq1_proof_thm5-4}), we can see that the event  $S_{1,F,w}(\bm{U}_n)>t_{\alpha,n}$ is asymptotically equivalent to $\overline{F}(a_{n,\gamma}S_{1,F,w}(\bm{U}_n))<\alpha_n$ for large enough $n$.
Therefore, 
$$\bm{P}\left(\overline{F}(a_{n,\gamma}S_{1,F,w}(\bm{U}_n))<\alpha_n\right)\sim \alpha_n$$
as $n\to\infty$.
This completes the proof of this theorem under Assumption \ref{assmp:geluktang} or \ref{assmp:kotang} with a nonrandom weight Assumption \ref{assump:w_equal} or \ref{assump:w_nonrandom} when $F$ satisfies a balance condition with $p_F=1$ and $q_F=0$.

Now let us further assume that $X_i$ are nonnegative. Extending the above result to UTAI, or Assumption \ref{assmp:UTAI}, is straightforward based on Theorem 3.1 of \cite{chen2009sums} and by mimicking the argument in the proof of Theorem 2 of \cite{bae2017note}. Note that Theorem 3.1 of \cite{chen2009sums} is written for equal weights and possibly nonidentical distributions for $X_i$ in $\mathcal{C}$. It is easy to see that if the distribution of $X_i$ is in $\mathcal{C}$, for any constant $c>0$, the distribution of $cX_i$ is also in $\mathcal{C}$. Therefore, 
their result also holds for Assumption \ref{assump:w_nonrandom}.

Since Assumption \ref{assmp:TAI} implies \ref{assmp:UTAI}, the asymptotic tail-precise property of $T_{1,F,w}(\bm{U}_n)$ also holds for Assumption \ref{assmp:TAI}.
\end{proof}

\begin{proof}[Proof of Theorem \ref{thm:ninfty_chen}]
Let $\alpha_n=C\epsilon_n n^{-\tau}$ for some $\epsilon_n\to0$ as $n\to\infty$ and positive real numbers $C$ and $\tau$ that do not depend on $n$.
The tail probability of the sum-combination statistic with the lower threshold $t_{\alpha,n}$ that satisfies $\overline{F}(t_{\alpha,n})=\alpha(\sum_{i=1}^n\bm{E}w_i^\gamma)^{-1}$
can be written as
$$\bm{P}(S_{1,F,w}(\bm{U}_n)>t_{\alpha,n})
=\bm{P}\left(\sum_{i=1}^n \frac{w_i}{n^{\frac{\tau}{\gamma}}(\sum_{j=1}^n \bm{E}w_j^\gamma)^{\frac{1}{\gamma}}}X_i>\overline{F}^{-1}(C\epsilon_n)\right).$$
Let $$\widetilde{w}_i=\frac{w_i}{n^{\frac{\tau}{\gamma}}(\sum_{j=1}^n\bm{E} w_j^\gamma)^{\frac{1}{\gamma}}}.$$

Consider $0<\gamma<1$. By Lemma \ref{lemma:wigamma_random}, $\sum_{i=1}^n\bm{E} w_i^\gamma\geq 1$, and therefore, $\bm{P}(0<\widetilde{w}_i<1)=1$. Lemma \ref{lemma2} implies that $\bm{E}\widetilde{w}_i^{q_1}\vee\bm{E}\widetilde{w}_i^{q_2}=\bm{E}\widetilde{w}_i^{q_1}$ for any $0<q_1<q_2$. For any positive real value $q_1<\gamma<1$,
$$\sum_{i=1}^n\bm{E}\widetilde{w}_i^{q_1}=\frac{\sum_{i=1}^n \bm{E}w_i^{q_1}}{n^{\frac{\tau q_1}{\gamma}}(\sum_{j=1}^n \bm{E}w_j^\gamma)^{\frac{q_1}{\gamma}}}\leq \frac{n^{1-q_1}}{n^{\frac{\tau q_1}{\gamma}}}=n^{1-q_1-\frac{\tau q_1}{\gamma}}$$
is uniformly bounded for all $n$ as long as $\tau\geq \frac{\gamma}{q_1}-\gamma$.

Now consider $\gamma\geq 1$. By taking the exponent $q_2>\gamma\geq1$, expectation, and the exponent $1/q_2$ on both sides, and taking the sum over $i=1,\ldots,n$,
$$\sum_{i=1}^n\left(\bm{E}\widetilde{w}_i^{q_2}\right)^{\frac{1}{q_2}}=\frac{\sum_{i=1}^n \left(\bm{E}w_i^{q_2}\right)^{\frac{1}{q_2}}}{n^{\frac{\tau}{\gamma}}(\sum_{j=1}^n \bm{E}w_j^\gamma)^{\frac{1}{\gamma}}}\leq \frac{n^{1-\frac{1}{q_2}}}{n^{\frac{\tau}{\gamma}}n^{\frac{1-\gamma}{\gamma}}}=n^{1-\frac{1}{q_2}-\frac{\tau}{\gamma}-\frac{1-\gamma}{\gamma}}=n^{2-\frac{1}{q_2}-\frac{\tau}{\gamma}-\frac{1}{\gamma}},$$
where the inequality is due to Lemma \ref{lemma:wigamma_random}. 
Therefore, 
$\sum_{i=1}^n\left\{\left(\bm{E}\widetilde{w}_i^{q_1}\right)^{\frac{1}{q_1}} \vee \left(\bm{E}\widetilde{w}_i^{q_2}\right)^{\frac{1}{q_2}} \right\}\leq \sum_{i=1}^n\left(\bm{E}\widetilde{w}_i^{q_2}\right)^{\frac{1}{q_2}}$ is uniformly bounded for all $n$
if $\tau\geq 2\gamma-\frac{\gamma}{q_2}-1$ when $\gamma\geq1$.

Note that when $\gamma=1$, the range of $2\gamma-\frac{\gamma}{q_2}-1$ is roughly $(0,1)$, allowing up to  $\alpha_n=o(n^{-0})=o(1)$. Similarly, $\alpha_n$ up to $o(n^{-(1-\gamma)})$ is allowed when $\gamma<1$ and 
$o(n^{-2(\gamma-1)})$  when $\gamma>1$, implying Assumption \ref{assump:alpha_randw}. 


The rest of the proof is similar to that in the Proof of Theorem \ref{thm:ninfty_gao}, using Proposition \ref{proposition:infinite_n_chen}.
\end{proof}

\begin{proof}[Proof of Theorem \ref{thm:ninfty_geng}]
Consider the same $\alpha_n$,  $t_{\alpha,n}$, and $\widetilde{w}_i$ as in the Proof of Theorem \ref{thm:ninfty_chen}. 
Consider $0<\gamma<1$. For any positive real value $\delta\in(0,\min\{\gamma,1-\gamma\})$, by Lemma \ref{lemma:wigamma_random},
$$\sum_{i=1}^n\bm{E}\widetilde{w}_i^{\gamma-\delta}=\frac{\sum_{i=1}^n \bm{E}w_i^{\gamma-\delta}}{n^{\frac{\tau (\gamma-\delta)}{\gamma}}(\sum_{j=1}^n \bm{E}w_j^\gamma)^{\frac{\gamma-\delta}{\gamma}}}\leq \frac{n^{1-(\gamma-\delta)}}{n^{\frac{\tau (\gamma-\delta)}{\gamma}}}=n^{1-(\gamma-\delta)-\frac{\tau (\gamma-\delta)}{\gamma}}$$
is uniformly bounded for all $n$ as long as $\tau\geq \frac{\gamma}{\gamma-\delta}-\gamma$.
Note that $\delta\in(0,\min\{(\gamma,1-\gamma\})$ is equivalent to $\gamma-\delta\in(\max\{0,2\gamma-1\},\gamma)$.

Now consider $\gamma\geq 1$. For some $\delta\in(0,\gamma)$, by H\"older's inequality,
\begin{equation}\label{eq:thm:geng:gamma1_UB}\begin{array}{ll}\sum_{i=1}^n \left(\bm{E}w_i^{\gamma-\delta}\right)^{\frac{1}{\gamma+\delta}}
&\leq \left[\sum_{i=1}^n\left\{\left(\bm{E}w_i^{\gamma-\delta}\right)^{\frac{1}{\gamma+\delta}}\right\}^{{\gamma+\delta}}\right]^\frac{1}{\gamma+\delta}\left(\sum_{i=1}^n1^{\frac{\gamma+\delta}{\gamma+\delta-1}}\right)^{\frac{\gamma+\delta-1}{\gamma+\delta}}
\\ &= \left(\sum_{i=1}^n\bm{E}w_i^{\gamma-\delta}\right)^{\frac{1}{\gamma+\delta}}n^{\frac{\gamma+\delta-1}{\gamma+\delta}}.
\end{array}\end{equation}
We first consider $\gamma>1$. Notice that the smaller $\delta$ is, the smaller the upper bound of (\ref{eq:thm:geng:gamma1_UB}) is. 
In particular, since  $\gamma>1$, it is possible to choose a positive constant $\delta$ such that $\gamma-\delta>1$. For such $\delta$, since $\gamma+\delta>\gamma-\delta>1$, Lemma \ref{lemma2} implies that $\bm{E}w_i^{\gamma+\delta}<\bm{E}w_i^{\gamma-\delta}<1$. Further by Lemma \ref{lemma1}, since $\gamma+\delta>1$, $\left(\bm{E}w_i^{\gamma-\delta}\right)^{\frac{1}{\gamma+\delta}}<\left(\bm{E}w_i^{\gamma+\delta}\right)^{\frac{1}{\gamma+\delta}}$.
Thus, the upper bound of (\ref{eq:thm:geng:gamma1_UB}) is, using Lemma \ref{lemma:wigamma_random},
$$\sum_{i=1}^n\left(\bm{E}\widetilde{w}_i^{\gamma-\delta}\right)^{\frac{1}{\gamma+\delta}}=\frac{\sum_{i=1}^n \left(\bm{E}w_i^{\gamma-\delta}\right)^{\frac{1}{\gamma+\delta}}}{n^{\frac{\tau(\gamma-\delta)}{\gamma(\gamma+\delta)}}(\sum_{j=1}^n \bm{E}w_j^\gamma)^{\frac{\gamma-\delta}{\gamma(\gamma+\delta)}}}\leq \frac{n^{\frac{\gamma+\delta-1}{\gamma+\delta}}}{n^{\frac{\tau(\gamma-\delta)}{\gamma(\gamma+\delta)}}n^{\frac{(1-\gamma)(\gamma-\delta)}{\gamma(\gamma+\delta)}}}=n^{\frac{2\gamma^2-2\gamma+\delta-\tau(\gamma-\delta)}{\gamma(\gamma+\delta)}},$$
which is uniformly bounded for any $n$ if $\tau\geq \frac{2\gamma^2-2\gamma+\delta}{\gamma-\delta}$. 

Now consider $\gamma=1$. In this case, $0<\gamma-\delta<1$ for any positive constant $\delta<\gamma$. Using Lemma \ref{lemma:wigamma_random}, the upper bound of (\ref{eq:thm:geng:gamma1_UB}) is $n^{\frac{1-\gamma+\delta}{\gamma+\delta}}n^{\frac{\gamma+\delta-1}{\gamma+\delta}}=n^{\frac{2\delta}{\gamma+\delta}}=n^{\frac{2\delta}{1+\delta}}$. 
Noting that $\sum_{j=1}^n\bm{E}w_j^\gamma=1$ when $\gamma=1$, we have
$$\sum_{i=1}^n\left(\bm{E}\widetilde{w}_i^{\gamma-\delta}\right)^{\frac{1}{\gamma+\delta}}=\frac{\sum_{i=1}^n \left(\bm{E}w_i^{\gamma-\delta}\right)^{\frac{1}{\gamma+\delta}}}{n^{\frac{\tau(\gamma-\delta)}{\gamma(\gamma+\delta)}}(\sum_{j=1}^n \bm{E}w_j^\gamma)^{\frac{\gamma-\delta}{\gamma(\gamma+\delta)}}}\leq \frac{n^{\frac{2\delta}{1+\delta}}}{n^{\frac{\tau(1-\delta)}{1+\delta}}}=n^{\frac{2\delta-\tau(1-\delta)}{1+\delta}},$$
which is uniformly bounded for any $n$ if $\tau\geq\frac{2\delta}{1-\delta}$.

For the $\gamma<1$ case, $\frac{\gamma}{q_1}-\gamma$ is decreasing in $q_1$, and therefore, $\alpha_n$ up to $o(n^{-(1-\gamma)})$ is allowed. For the other two cases, $\gamma=1$ and $\gamma>1$, $\frac{2q_2}{1-q_2}$ and $\frac{2\gamma^2-2\gamma+q_3}{\gamma-q_3}$ are increasing in $q_2$ and $q_3$, respectively. This can be easily seen by noticing that the first derivative of both functions is always positive since $q_2$ and $q_3$ are not equal to 1.
Therefore,$\alpha_n$ up to $o(n^{-\frac{0}{1}})=o(1)$ is allowed when $\gamma=1$ and $o(n^{-2(\gamma-1)})$ is allowed when $\gamma>1$.

The rest of the proof is similar to that in the Proof of Theorem \ref{thm:ninfty_gao}, using Proposition \ref{proposition:infinite_n_geng}.
\end{proof}

\begin{proof}[Proof of Theorem \ref{thm:max_iid_gnedenko}]
We first note that, since $\overline{F}$ is continuous, and therefore, strictly decreasing, $c_n=\overline{F}^{-1}(n^{-1})\uparrow\infty$ as $n\uparrow\infty$.
Since $\overline{F}(c_n)=1/n$ and $\overline{F}(nt_{\alpha,n})=\alpha/n$,
$$\overline{F}(nt_{\alpha,n})=\alpha\overline{F}(c_n)\sim\overline{F}(\alpha^{-\frac{1}{\gamma}}c_n)$$
as $c_n\uparrow\infty$, where the last approximation is due to the asymptotic scale invariance of a regularly varying distribution. Therefore, $nt_{\alpha,n}\sim\alpha^{-\frac{1}{\gamma}}c_n$, or equivalently,
$$c_n^{-1}\sim n^{-1}t_{\alpha,n}^{-1}\alpha^{-\frac{1}{\gamma}}$$
as $n\uparrow\infty$, and further,
$$\bm{P}\left(S_{3,F,w}(\bm{U}_n)>t_{\alpha,n}\right)=\bm{P}\left(n^{-1}\bigvee_{i=1}^n X_i>t_{\alpha,n}\right)\sim\bm{P}\left(c_n^{-1}\bigvee_{i=1}^n X_i>\alpha^{-\frac{1}{\gamma}}\right).$$
From (\ref{eq:max_ninfty}),
$$\bm{P}\left(c_n^{-1}\bigvee_{i=1}^n X_i<\alpha^{-\frac{1}{\gamma}}\right)\sim \exp\left(-(\alpha^{-\frac{1}{\gamma}})^{-\gamma}\right)=\exp(-\alpha)\sim 1-\alpha$$
as long as $\alpha=\alpha_n\downarrow0$ as $n\uparrow\infty$, which finishes the proof.
\end{proof}

\begin{proof}[Proof of Theorem \ref{thm:ninfty_max}]
Since $X_i=\overline{F}^{-1}(U_i)$ are independent,  
$\bm{P}(\bigvee_{i=1}^nw_iX_i<t)=\prod_{i=1}^n\bm{P}(X_i<tw_i^{-1})$. Taking log on both sides, we have
$$\log\left(\bm{P}\left(\bigvee_{i=1}^nw_iX_i<t\right)\right)=\sum_{i=1}^n\log\bm{P}\left(X_i<tw_i^{-1}\right)=\sum_{i=1}^n\log F(tw_i^{-1}).$$
Note that $tw_i^{-1}\geq t\to\infty$ as $t\to\infty$ for all $i=1,\ldots,n$, regardless of the value of $n$.
Using the relationship $\log x\uparrow (x-1)$ as $x\uparrow1$,  we have $\log F(tw_i^{-1})\sim F(tw_i^{-1})-1=-\overline{F}(tw_i^{-1})$ as $t\to\infty$ for all $i=1,\ldots,n$, no matter what $n$ is. Therefore,
$$\bm{P}\left(\bigvee_{i=1}^nw_iX_i>t\right) \sim  1-\exp\left(-\sum_{i=1}^n\overline{F}(tw_i^{-1})\right)$$ as $t \to\infty$ for any $n$. In addition, $\min_i t w_i^{-1}\to\infty$ implies $\overline{F}(tw_i^{-1})\to0$  as $t\to\infty$, and hence, 
 the relationship $1-\exp(-x)\sim x$ as $x\to0$ implies 
$$\bm{P}\left(\bigvee_{i=1}^nw_iX_i>t\right) \sim \sum_{i=1}^n\overline{F}(tw_i^{-1})$$
as $t\to\infty$ for any $n$. Therefore,
$$\bm{P}\left(\bigvee_{i=1}^\infty w_iX_i>t\right) \sim \sum_{i=1}^\infty\overline{F}(tw_i^{-1}),$$
which is equivalent to equation (\ref{thm:eq:max_ninfty}).
\end{proof}

\begin{proof}[Proof of Theorem \ref{thm:gclt}]
From the GCLT in equations (\ref{eq:gcltgamma1}) and (\ref{eq:gcltgamma2}), 
the tail probability of the sum-combining function can be approximated as, when $U_i$ are iid and $n$ is large enough,
\begin{equation*}\label{eq:gclt}\begin{array}{lll}
\bm{P}(T_{1,F_\gamma,w}(\bm{U}_n)<\alpha_n)&=&\bm{P}(\overline{F}(a_{n,\gamma}S_{1,F_\gamma,w}(\bm{U}_n))<\alpha_n)=\bm{P}\left(n\left(\sum_{i=1}^nU_i^{-\frac{1}{\gamma}}\right)^{-\gamma}<\alpha_n\right)
\\&=&\bm{P}\left(\sum_{i=1}^nU_i^{-\frac{1}{\gamma}}>\alpha_n^{-\frac{1}{\gamma}}n^{\frac{1}{\gamma}}\right)
\\&\sim&\left\{\begin{array}{ll}
\bm{P}\left(S(1,1,1,0)>\frac{\alpha_n^{-1}-\log n}{\pi /2}\right),&\gamma=1
\\ \bm{P}\left(S(\gamma,1,1,\tan(\pi\gamma/2))>\alpha_n^{-\frac{1}{\gamma}}\left\{\frac{2}{\pi}\Gamma(\gamma)\sin\left(\frac{\pi\gamma}{2}\right)\right\}^{\frac{1}{\gamma}}\right), &\gamma\in(0,1),
\end{array}\right.
\end{array}\end{equation*}
where $S(\gamma,\beta,\sigma,\mu)$ indicates a random variable that follows $S(\gamma,\beta,\sigma,\mu;0)$.
Consider $\gamma\in(0,1)$. The proof is straightforward 
using the tail approximation of a stable distribution \cite[Theorem 1.2]{nolan2020univariate} 
as long as $\alpha_n=o(1)$. 

Now we consider $\gamma=1$. Since $\alpha_n=o(1/\log n)$, we have  $\alpha_n^{-1}-\log n\to\infty$ as $n\to\infty$. Therefore, following Theorem 1.2 of \cite{nolan2020univariate}, 
$\bm{P}\left(S(1,1,1,0)>\frac{\alpha_n^{-1}-\log n}{\pi /2}\right)\sim(\alpha_n^{-1}-\log n)^{-1}$ as $n\to\infty$. 
\end{proof}

\end{document}